\documentclass[]{amsart}
\usepackage{amsfonts}

\usepackage{amssymb}
\usepackage{amsmath}
\usepackage{amsthm}
\usepackage{amscd}
\usepackage{graphicx}
\usepackage{tikz}
\usepackage{amsmath,amssymb,amsfonts,amsthm,tikz,caption,subcaption}
\usepackage{psfrag}
\usepackage{color}
\usepackage{enumerate}

\newtheorem{theorem}{Theorem}[section]

\newtheorem{lemma}{Lemma}[section]
\newtheorem{proposition}[theorem]{Proposition}

\definecolor{darkgreen}{cmyk}{1,0,1,.2}

\renewcommand\paragraph[1]{\medskip\textbf{#1} }

\begin{document}
\title[the sister of picard modular group]{generators  of the sister of Euclidean Picard modular  group}
\author{BaoHua Xie} \address{College of Mathematics and Econometrics \\Hunan University\\ Changsha, 410082, P.R.China}
\email{xiexbh@hnu.edu.cn}



\begin{abstract}
The sister of Eisenstein-Picard modular group was described in \cite{p1}. In this paper we give a similar definition of the sister of the Euclidean-Picard
modular group and find its generators by using a geometric method.
\end{abstract}

\maketitle

  \section{Introduction}  
  An important and difficult problem  in complex hyperbolic geometry is to determine  the generators systems and finite presentations for the lattices in $\mathbf{ PU}(2,1)$.  
 The most frequent method is to construct a fundamental domain and use the Poincar\'e's polyhedron theorem. It should be pointed out that, even though fundamental domains  of arithmetic lattices are well known to exist,their explicit determination in the context of the non-constant curvature setting of complex hyperbolic space is a difficult task. There is also a simple algorithm for one to get the generators of some lattices. In particular, we do not need to know the fundamental domain for the action of  lattice on the complex hyperbolic plane. See the works \cite{fflp,wxx, xwj1}. However, it is hard to get more information of the lattices by this method.

  There are not so many examples of explicit arithmetic lattices or non- arithmetic lattices of   $\mathbf{ PU}(2,1)$, See, for example, the works of
Mostow\cite{m},  Falbel and Parker\cite{fp}, Parker\cite{p2}, Deraux, Falbel and Paupert\cite{dfp}, Deraux, Parker and Paupert\cite{dpp}.  Perhaps the first
example for the complex hyperbolic space in two complex dimensions was due to Picard \cite{pi1,pi2}, which we now call Picard modular group.

  Let $\mathcal{O}_d$ be the ring of integers in the quadratic imaginary number field $\mathbb{Q}(i\sqrt{d})$, where $d$ is
a positive square-free integer. If $d\equiv1,2\ (mod\ 4)$, then $\mathcal{O}_d=\mathbb{Z}[i\sqrt{d}]$ and if $d\equiv3\ (mod\ 4)$, then $\mathcal{O}_d=\mathbb{Z}[\omega_d]$, where $\omega_d=(1+i\sqrt{d})/2 $. The Picard modular groups are subgroups of $\mathbf{ PU}(2,1)$ with entries in  $\mathcal{O}_d$ and written $\mathbf{ PU}(2,1;\mathcal{O}_d)$. The Picard modular groups $\mathbf{ PU}(2,1;\mathcal{O}_d)$ can be considered as the natural algebraic generalization of the classical modular group  $\mathbf{ PSL}(2,\mathbb{Z})$.

For $d=1,2,3,7,11$ the rings $\mathcal{O}_d$ have Euclidean algorithms and the corresponding groups $\Gamma^{(d)}=\mathbf{ PU}(2,1;\mathcal{O}_d)$ are the Euclidean Picard modular groups. As is to be expected, these
are much closer in properties to the modular group than in the non-Euclidean cases.  Of particular interest are the group $\Gamma^{(1)}$ and $\Gamma^{(3)}$, which are called the 
Gauss-Picard  modular group and  the Eisenstein-Picard modular group.  However, the explicit algebraic or geometric properties such as generators, fundamental domains and presentations are still unknown in all but very few cases. More specifically, presentation and fundamental domains have been obtained for $\Gamma^{(d)}$ when $d=1$   and $d=3$.
 More recently, generators for $\mathbf{ PU}(2,1;\mathcal{O}_d)$ with $ d=2, 7,11$ were given in \cite{z2}.
 
  In this paper, we will use the geometric arguments in \cite{fp, xwj2, z2} to obtain the generators of the sister of some Euclidean-Picard modular groups. In \cite{p1}, the sister of the Einsenstein-Picard modular group $\mathbf{ PU}(2,1;\mathcal{O}_3)$  was defined explicitly. It's fundamental domain and generators was obtained in \cite{z1}. Similarly, we can also define the sister of the other Picard modular groups. Let $\Gamma^{(d)}_s$ be the sister of Euclidean Picard modular groups  $\mathbf{ PU}(2,1;\mathcal{O}_d)$. We claim that the same feature shared by Euclidean Picard modular groups is that the quotient  $\mathbf{H}^2_{\mathbb{C}}/\Gamma^{(d)}_s$ has only one cusp. This is achieved by showing that  
$\Gamma^{(d)}$ and $\Gamma^{(d)}_s$ has a subgroup with the same finite index.
 
  We will start with finding suitable generators of the stabilizer of infinity of $\Gamma^{(d)}_s$ and then construct  a fundamental domain for the stabilizer acting on the
 boundary of complex hyperbolic space. We also obtain a presentation of the isotropy subgroup fixing infinity by analyzing  the combinatorics of the fundamental domain.
  Finally, we will determine some isometric spheres such that the intersection of these isometric spheres and the fundamental domain for the stabilizer has one cusp.
 
 Our  main results are the  following three theorems.

\begin{theorem} The group $\Gamma^{(2)}_s$ is generated by the elements
$$I^{(2)}_0=\left[\begin{array}{ccc}0&0&i/\sqrt{2}\\0&1&0\\i\sqrt{2}&0&0\end{array}\right], R^{(2)}_1=\left[\begin{array}{ccc}1&0&0\\0&-1&0\\0&0&1\end{array}\right],
 R^{(2)}_2=\left[\begin{array}{ccc}1&2&-2\\0&-1&2\\0&0&1\end{array}\right],$$
$$ R^{(2)}_3=\left[\begin{array}{ccc}1&-i\sqrt{2}&-1\\0&-1&i\sqrt{2}\\0&0&1\end{array}\right]
\quad\text{and}\quad T^{(2)}=\left[\begin{array}{ccc}1&0&i/\sqrt{2}\\0&1&0\\0&0&1\end{array}\right]
.$$
\end{theorem}

\medskip

\begin{theorem}  The group $\Gamma^{(7)}_s$ is generated by the elements
$$I^{(7)}_0=\left[\begin{array}{ccc}0&0&i/\sqrt{7}\\0&1&0\\i\sqrt{7}&0&0\end{array}\right], R^{(7)}_1=\left[\begin{array}{ccc}1&0&0\\0&-1&0\\0&0&1\end{array}\right],
 R^{(7)}_2=\left[\begin{array}{ccc}1&1& i\omega_7/\sqrt{7}\\0&-1&1\\0&0&1\end{array}\right],$$
 $$
 R^{(7)}_3=\left[\begin{array}{ccc}1&\overline{\omega_7}&-1\\0&1&\omega_7\\0&0&1\end{array}\right]\quad\text{and}\quad T^{(7)}=\left[\begin{array}{ccc}1&0&i/\sqrt{7}\\0&1&0\\0&0&1\end{array}\right]
.$$
\end{theorem}

\medskip

\begin{theorem}  The group $\Gamma^{(11)}_s$ is generated by the elements
$$I^{(11)}_0=\left[\begin{array}{ccc}0&0&i/\sqrt{11}\\0&1&0\\i\sqrt{11}&0&0\end{array}\right], R^{(11)}_1=\left[\begin{array}{ccc}1&0&0\\0&-1&0\\0&0&1\end{array}\right],
 R^{(11)}_2=\left[\begin{array}{ccc}1&1&i\omega_{11}/\sqrt{11}\\0&-1&1\\0&0&1\end{array}\right],$$
 $$R^{(11)}_3=\left[\begin{array}{ccc}1&\overline{\omega_{11}}&-3/2+i/2\sqrt{11}\\0&-1&\omega_{11}\\0&0&1\end{array}\right]\quad\text{and}\quad T^{(11)}=\left[\begin{array}{ccc}1&0&i/\sqrt{11}\\0&1&0\\0&0&1\end{array}\right]
.$$
\end{theorem}

\section{Complex hyperbolic space}
\label{sec-complex}
\subsection{The Siegel domain}

\medskip

 A general reference for complex hyperbolic geometry is \cite{g}.  Let $\mathbb{C}^{2,1}$ denote 
the complex vector space of dimension  $3$, equipped with a non-degenerate
Hermitian form of signature $(2,1)$.There are several such forms. If we use the  second Hermitian form on $\mathbb{C}^{2,1}$, for column vectors 
$\mathbf{ z}=(z_1,z_2,z_3)^t$ and $\mathbf{ w}=(w_1,w_2,w_3)^t$,
$$\langle\mathbf{  z},\mathbf{ w}\rangle=\mathbf{ w}^*J\mathbf{ z}=z_1\overline{w_3}+z_2\overline{w_2}+z_3\overline{w_1},$$ where 
\begin{equation*}
J=\left(\begin{array}{ccc}
0 & 0& 1\\
0 & 1& 0\\
1 & 0& 0
\end{array}\right)
\end{equation*}
and $\mathbf{ w}^*$ is the Hermitian transpose of $\mathbf{ w}$.

The projective model of complex hyperbolic space $\mathbf{H}_{\mathbb{C}}^2$ is defined to be the collection of negative
lines in $\mathbb{C}^{2,1}$, namely, those points $\mathbf{z}$ satisfying $\langle \mathbf{z},\mathbf{z}\rangle<0$.

We mainly take the {\it Siegel domain} $\mathfrak{S}$ as a upper half-space model for the complex hyperbolic space, that is given by
$$\mathfrak{S}=\{(z_1,z_2)\in\mathbb{C}^2:2\Re ez_1+|z_2|^2<0\}.$$ The boundary of the {\it Siegel domain} $\mathfrak{S}$ is identified with the one-point compactification of the Heisenberg group. The Heisenberg group $\mathfrak{R}$ is $\mathbb{C}\times \mathbb{R}$ with the group law
\begin{equation*}(\zeta_1,t_1)\diamond(\zeta_2,t_2)=(\zeta_1+\zeta_2,t_1+t_2+2\Im m(\zeta_1\bar{\zeta}_2)).\end{equation*}
There is a canonical projection from $\mathfrak{R}$ to $\mathbb{C}$ called {\it vertical projection} and denoted by $\Pi$, given by $\Pi: (\zeta,t)\longmapsto \zeta.$

The Cygan metric on $\mathfrak{R}$ is given by
$$
\rho_0\left((\zeta_1,t_1),(\zeta_2,t_2)\right)=\left||\zeta_1-\zeta_2|^2-it_1+it_2-2i\Im m(\zeta_1\bar{\zeta}_2)\right|.
$$

We can extend the Cygan metric to an incomplete metric on
$\bar{\mathfrak{S}}-\{\infty\}$ as follows
$$\tilde{\rho}_0((\zeta_1,t_1,u_1),(\zeta_2,t_2,u_2))=\left||\zeta_1-\zeta_2|^2+|u_1-u_2|-it_1+it_2-2i\Im m(\zeta_1\bar{\zeta}_2)\right|.$$

Any point  $\mathbf{ p}\in\mathfrak{S}$  admits a unique lift to $\mathbb{C}^{2,1}$ of the following form, called its standard lift
\begin{equation}\label{eq:2-1}\mathbf{ p}=\left[\begin{array}{c}(-|\zeta|^2-u+it)/2\\ \zeta\\1\end{array}\right]\end{equation}
with $(\zeta,t,u)\in \mathbb{C}\times \mathbb{R}\times]0,\infty[.$   In particular,
\begin{equation*}q_{\infty}=\left[\begin{array}{c}1\\0\\0\end{array}\right].\end{equation*}
Then the triple $(\zeta,t,u)$ is called the horospherical coordinates of $\mathbf{p}$  and $\mathfrak{S}=\mathfrak{R}\times\mathbb{R}_+$ and $\partial\mathfrak{S}=(\mathfrak{R}\times\{0\})\cup\{q_{\infty}\}$.
\medskip

\subsection{Complex hyperbolic isometries}

\medskip

Let $\mathbf{U}(2,1)$ be the group of matrices that are unitary with respect to the form $\langle .,.\rangle$. The group of holomorphic isometries of complex hyperbolic space is the projective unitary group
$\mathbf{PU}(2,1)=\mathbf{U}(2,1)/\mathbf{U}(1)$, with a natural identification $\mathbf{U}(1)=\{e^{i\theta}I,\theta\in[0,2\pi)\}.$ We now describe the action of the stabilizer of $q_\infty$ on the Heisenberg group.

The Heisenberg group acts on itself by {\it Heisenberg translations}. For $(\tau,v)\in\mathfrak{R}$, this is
\begin{equation*}T_{(\tau,v)}:(z,t)\mapsto(z+\tau,t+v+2\Im m(\tau\bar{z}))=(\tau,v)\diamond(z,t).\end{equation*}
Heisenberg translation by $(0,v)$ for any $v\in\mathbb{R}$ is called {\it vertical translation} by $v$.

The unitary group $\mathbf{U}(1)$ acts on the Heisenberg group by {\it Heisenberg rotations}. For $e^{i\theta}\in \mathbf{ U}(1)$, the rotation fixing $q_0=(0,0,0)$ is given by \begin{equation*}R_{\theta}:(z,t)\mapsto(e^{i\theta}z,t).\end{equation*} 
For $r\in\mathbb{R}_+$, {\it Heisenberg dilation} by $r$ fixing $q_\infty$ and $q_0=(0,0,0)\in\partial\mathbf{H}^2_{\mathbb{C}}$ is given by
\begin{equation*}D_r:(z,t)\mapsto(r z,r^2 t).\end{equation*}

The stabilizer of $q_{\infty}$ in $\mathbf{PU}(2,1)$ is generated by all Heisenberg translations, rotations and dilations. 
The matrices of the 3 kinds of isometries are
$$T_{(z,t)}=\left[\begin{array}{ccc}
1&-\overline{z}&-(|z|^2-it)/2\\
0&1&z\\
0&0&1
\end{array}\right], R_{\theta}=\left[\begin{array}{ccc}
1&0&0\\
0&e^{i\theta}&0\\
0&0&1
\end{array}\right], D_{r}=\left[\begin{array}{ccc}
r&0&0\\
0&1&0\\
0&0&1/r
\end{array}\right].$$

Note that Heisenberg translations and rotations preserve each horosphere  based at $q_{\infty}$, whereas Hersenberg dilations  permute horosphere based at $q_{\infty}$.
For this reason the group generated by all Heisenberg translations and rotations, which is the semidirect product $\mathbf{ U}(1)\ltimes\mathfrak{R}$, is called the {\it Heisenberg isometry group} $Isom(\mathfrak{R})$.
In addition, the group $Isom(\mathfrak{R})$ consists exactly of those matrices of the following form:
$$\left[\begin{array}{ccc}
1&-\overline{z}e^{i\theta}&-(|z|^2-it)/2\\
0&e^{i\theta}&z\\
0&0&1
\end{array}\right].$$

Recall from \cite{fp} that the  the exact sequence
\begin{equation*}0\longrightarrow\mathbb{R}\longrightarrow\mathfrak{R}\stackrel{\Pi}{\longrightarrow}\mathbb{C}\longrightarrow0,\end{equation*}
induces the exact sequence 
\begin{equation}\label{eq:2-2}0\longrightarrow\mathbb{R}\longrightarrow  Isom(\mathfrak{R})\stackrel{\Pi_{*}}{\longrightarrow}  Isom(\mathbb{C})\longrightarrow1.\end{equation}

Explicitly, $\Pi_{*}\left(Isom(\mathfrak{R})\right)=\left[ \begin{array}{cc}
 e^{i\theta}&\zeta_0\\0&1
\end{array}\right]$, acting on $\mathbb{C}$ by $w\rightarrow e^{i\theta}w+z$ and $Ker(\Pi_{*})$ consist of Hersenberg vertical translation.
\medskip

\subsection{Isometric spheres}

\medskip

Given an element $G\in \mathbf{PU}(2,1)$ satisfying $G(q_{\infty})\neq
q_{\infty}$, we define the isometric sphere of $G$ to be the
hypersurface
$$\left\{\mathbf{z}\in\mathbf{H}^2_{\mathbb{C}}:|\langle \mathbf{z},q_{\infty}\rangle|=|\langle \mathbf{z},G^{-1}(q_{\infty})\rangle|\right\}.$$

We often consider  spheres with respect to the Cygan metric. The 
Cygan sphere of radius $r\in  \mathbb{R}_{+}$ and centre $(z_0,t_0,0)\in \mathfrak{R}$ is given
by
$$\left\{(z,t,u)\in \mathbf{H}^2_{\mathbb{C}} :\left||z-z_0|^2+u+it-it_0+2i\Im m(z\bar{z}_0)\right|=r^2\right\}.$$
Note that Cygan spheres are always convex.

If $G\in \mathbf{PU}(2,1)$ has the matrix form \begin{equation}\label{eq:2-6}
\left[\begin{array}{ccc}
z_{11}&z_{12}&z_{13}\\
z_{21}&z_{22}&z_{23}\\
z_{31}&z_{32}&z_{33}
\end{array}\right]
,\end{equation} then $G(q_{\infty})\neq q_{\infty}$ if and
only if $z_{31}\neq0$. The isometric sphere of $G$ is a  Cygan sphere of radius 
$r=\sqrt{2/|z_{31}|}$ and centre $G^{-1}(q_{\infty})$, which in
horospherical coordinates is
$$(z_0,t_0,0)=(\bar{z_{32}}/\bar{z_{31}},2\Im m(\bar{z_{33}}/\bar{z_{31}}),0).$$

\section{the sister groups of Picard modular groups}

Let $\mathcal{O}_d$ be the ring of integers in the quadratic imaginary number field $\mathbb{Q}(i\sqrt{d})$, where $d$ is
a positive square-free integer. If $d\equiv1,2\pmod 4$, then $\mathcal{O}_d=\mathbb{Z}[i\sqrt{d}]$ and if
 $d\equiv3 \pmod 4$, then $\mathcal{O}_d=\mathbb{Z}[\omega_d]$, where $\omega_d=(1+i\sqrt{d})/2   $. The group $\Gamma^{(d)}=\mathbf{PU}(2,1;\mathcal{O}_d)$
 is called {\it Euclidean Picard modular group} if the ring $\mathcal{O}_d$ is Euclidean, namely, the rings $\mathcal{O}_1,\mathcal{O}_2,\mathcal{O}_3,\mathcal{O}_7,\mathcal{O}_{11}$.

The sister of the Eisenstein-Picard modular group $\mathbf{PU}(2,1;\mathcal{O}_3)$ was  defined explicitly in \cite{jor}.
In fact, we find that the similar definition of sister groups can be extended to the other Picard modular groups. In this paper, we will only consider the case $d=2,7,11.$
 Let $\Gamma_s^{(d)}$ be  the collection of all elements of $\mathbf{PU}(2,1)$ that, when written in the form (3), have $z_{11},z_{12},z_{13}(i\sqrt{d}), z_{21}/(i\sqrt{d}), z_{22}, z_{23}, z_{31}/(i\sqrt{d})$,
 and  $z_{33}$ all in $\mathcal{O}_d$. That is, in the case of $d=3 \pmod 4$
 
 \begin{enumerate}[(a)]
\item $z_{13}=x_{13}/2+iy_{13}/2\sqrt{d}$, where $x_{13}$ and $y_{13}$ are integers of the same parity; 
\item $z_{jk}=x_{jk}/2+i\sqrt{d}y_{jk}/2$ for all other $jk$, where $x_{jk}$ and $y_{jk}$ are integers of the same parity;
\item$x_{21}, x_{31},x_{32}$ are all divisible by $d$
\end{enumerate}

and  in the case of $d=2\pmod4$
 \begin{enumerate}[(a)]
\item $z_{13}=x_{13}+iy_{13}/\sqrt{2}$, where $x_{13}$ and $y_{13}$ are integers;
\item $x_{21}, x_{31},x_{32}$ are all divisible by $2$.
\end{enumerate}
It is simple to check that $\Gamma_s^{(d)}$ is a group. It is also discrete as $\mathcal{O}_d$ is discrete in $\mathbb{C}$. We will show that the sister groups $\Gamma_s^{(d)}$ 
defined above commensurable with the Picard modular group  $\Gamma^{(d)}$. Let $H^{(d)}=\Gamma_s^{(d)}\cap \Gamma^{(d)}$ be the intersection of $\Gamma_s^{(d)}$ and 
$\Gamma^{(d)}$. One will see that $H^{(d)}$ has the same index in $\Gamma_s^{(d)}$ and 
$\Gamma^{(d)}$. We will prove this for $d=2,7,11$ case by case.

For the case $d=2$, we have

\begin{lemma}$H^{(2)}$ has index three in $\Gamma_s^{(2)}$ and 
$\Gamma^{(2)}$.
\end{lemma}
\begin{proof}
First, we note that $H^{(2)}$ consists of all matrices in $\mathbf{PU}(2,1)$, which written in the form $(3)$, have $z_{jk}=x_{jk}+i\sqrt{2}y_{jk}$,
where $x_{jk}$ and $y_{jk}$ are integers and also $x_{21},x_{31},$ and $x_{32}$ are all divisible by $2$.

We decompose $\Gamma_s^{(2)}$ and 
$\Gamma^{(2)}$  into three $H^{(2)}$-cosets as follows. We claim that

$$\Gamma^{(2)}=H^{(2)}\cup g_1 H^{(2)}\cup g_2 H^{(2)}$$

where
$$g_1=\left[\begin{array}{ccc}
0&0&1\\
0&-1&0\\
1&0&0
\end{array}\right],  g_{2}=\left[\begin{array}{ccc}
1&0&0\\
-i\sqrt{2}&-1&0\\
-1&i\sqrt{2}&1
\end{array}\right]. $$

Let $g\in \Gamma^{(2)}$ be written in the form $(3)$. As $g\in \mathbf{PU}(2,1)$, we have
\begin{equation}
\overline{z_{11}}z_{31}+z_{11}\overline{z_{31}}+|z_{21}|^2=0
\end{equation}
and
\begin{equation}
\overline{z_{31}}z_{33}+z_{31}\overline{z_{33}}+|z_{32}|^2=0.
\end{equation}

Writing out $z_{jk}$ in terms of real and imaginary parts in $(5), (6)$ and considering congruence modulo $2$,  we obtain
\begin{equation}
x_{21}^2=-x_{11}x_{31} \pmod 2
\end{equation}
and
\begin{equation}
x_{32}^2=-x_{31}x_{33} \pmod 2.
\end{equation}

We need to show that for all $g=(z_{jk})$ in $\Gamma^{(2)}$ there is an $m\in\{0,1,2\}$ so that
$g_m^{-1}g=(z_{jk}')$ is in $H^{(2)}$($g_0$ is the identity). In other words, $x_{21}',   x_{31}', $ and $x_{32}'$
are all divisible by $2$. We begin by finding $g_m^{-1}g$ so that $2\mid x_{21}',   2\mid x_{31}'$. This is 
sufficient since from $(8)$,  we see that if $2 \mid x_{31}'$, then so is $x_{32}'$.

If $2\mid x_{21}$, then from (8) we see that either $x_{31}$  or $x_{11}$ is also divisible by $2$. In the first case
$g$  has $2\mid x_{21}$  and $2\mid x_{31}$  as required. Now $g_1^{-1}g$  has $x'_{21}=-x_{21}$  and  $x'_{31}=x_{11}$.
Thus if $2\mid x_{21}$  and $2\mid x_{11}$, then $g_1^{-1}g$  has $2\mid x'_{21}$ and $2\mid x'_{31}$.

If $2\nmid x_{21}$, then from  $(7)$ we see that $x_{11}$ and $x_{31}$ can not be divisible by $2$ either. In other words, 
$x_{11}\equiv x_{21}\equiv 1\pmod 2$.  One see that $g_2^{-1}g$  has $x'_{31}=-x_{11}-2y_{21}+x_{31}$ divisible by $2$.

Now we consider $\Gamma_s^{(2)}$. We claim that
$$\Gamma_s^{(2)}=H^{(2)}\cup g'_1 H^{(2)}\cup g'_2 H^{(2)}$$
where
$$g'_1=\left[\begin{array}{ccc}
0&0&i/\sqrt{2}\\
0&1&0\\
i\sqrt{2}&0&0
\end{array}\right],  g'_{2}=\left[\begin{array}{ccc}
1&0&i/\sqrt{2}\\
0&1&0\\
0&0&1
\end{array}\right].$$

Write a general element of $\Gamma_s^{(2)}$ in the form $(3)$, where $z_{13}=x_{13}+iy_{13}/\sqrt{2}$ and the other entries are as before. In order
to show that an element of $\Gamma_s^{(2)}$ is in $H^{(2)}$, we must show that $2\mid y_{13}$.
If $2\mid y_{13}$, then $g$ is in $H^{(2)}$. If $2\mid x_{33}$, then $g'^{-1}_1g$  has $y'_{13}= -x_{33}$ divisible by $2$.  So it is in $H^{(2)}$.
If $y_{13}\equiv x_{33}\equiv 1\pmod 2$.  One see that $g'^{-1}_2g$  has $y'_{13}=y_{13}-x_{33}$ divisible by $2$.
\end{proof}

For the case $d=7,11$, we have

\begin{lemma}$H^{(d)}$ has  index  $d+1$ both in $\Gamma_s^{(d)}$ and 
$\Gamma^{(d)}$.
\end{lemma}
\begin{proof}
First, we note that $H^{(d)}$ consists of all matrices in $\mathbf{ PU}(2,1)$, which written in the form $(6)$, have $z_{jk}=x_{jk}/2+i\sqrt{d}y_{jk}/2$,
where $x_{jk}$ and $y_{jk}$ are integers of the same parity, and also $x_{21},x_{31},$ and $x_{32}$ are all divisible by $d$.

We decompose $\Gamma_s^{(d)}$ and 
$\Gamma^{(d)}$  into $d+1$ $H^{(d)}$-cosets as follows. We claim that

$$\Gamma^{(d)}=H^{(d)}\bigcup_{i=1} ^{d}g_i H^{(d)}$$

where
$$g_1=\left[\begin{array}{ccc}
0&0&1\\
0&-1&0\\
1&0&0
\end{array}\right],  g_{2k}=\left[\begin{array}{ccc}
1&0&0\\
k&1&0\\
-k^2\omega_d&-k&1
\end{array}\right],  g_{2k+1}=\left[\begin{array}{ccc}
1&0&0\\
-k&1&0\\
-k^2\omega_d&k&1
\end{array}\right]  $$
where $k=1,2,\cdots, (d-1)/2$.

Let $g\in \Gamma^{(d)}$ be written in the form $(3)$. As $g\in \mathbf{PU}(2,1)$, we have
\begin{equation}
\overline{z_{11}}z_{31}+z_{11}\overline{z_{31}}+|z_{21}|^2=0
\end{equation}
and
\begin{equation}
\overline{z_{31}}z_{33}+z_{31}\overline{z_{33}}+|z_{32}|^2=0.
\end{equation}

Writing out $z_{jk}$ in terms of real and imaginary parts in $(9), (10)$ and considering congruence modulo $7$,  we obtain
\begin{equation}
x_{21}^2=(d-1)x_{11}x_{31} \pmod d
\end{equation}
and
\begin{equation}
x_{32}^2=(d-1)x_{31}x_{33} \pmod d.
\end{equation}

We need to show that for all $g=(z_{jk})$ in $\Gamma^{(d)}$ there is an $m\in\{0,1,2,\cdots d\}$ so that
$g_m^{-1}g=(z_{jk}')$ is in $H^{(d)}$($g_0$ is the identity). In other words, $x_{21}',   x_{31}', $ and $x_{32}'$
are all divisible by $d$. We begin by finding $g_m^{-1}g$ so that $x_{21}',   x_{31}'$ are divisible by $d$. This is 
sufficient since from $(12)$,  we see that if $d\mid x_{31}'$, then so is $x_{32}'$.

If $d\mid x_{21}$, then from (11) we see that either $d\mid x_{31}$  or $d\mid x_{11}$. In the first case
$g$  has $d\mid x_{21}$  and $d\mid x_{31}$ as required. Now $g_1^{-1}g$  has $x'_{21}=-x_{21}$  and  $x'_{31}=x_{11}$.
Thus if $d\mid x_{21}$  and $d\mid x_{11}$, then $g_1^{-1}g$  has $d\mid x'_{21}$ and $d\mid x'_{31}$.

If $x_{21}$ is not divisible by $d$, then from  $(11)$ we see that $x_{11}$ and $x_{31}$ can not be divisible by $d$ either.
There are some cases.

Case 1:  If $kx_{11} \equiv  x_{21} \equiv \pm1,\pm 2,\cdots, \pm (d-1)/2 \pmod d$, then $g_{2k}^{-1}g$ has $x_{21}'=x_{21}-kx_{11}$ divisible by $d$ and $x_{11}'=x_{11}$
not divisible by $d$. From $(11)$ we see that this means $g_{2k}^{-1}g$ has $x_{31}'$ divisible by $d$ as well.

Case 2:    If $kx_{11} \equiv  -x_{21} \equiv \pm1,\pm 2,\cdots, \pm (d-1)/2 \pmod d$, then $g_{2k+1}^{-1}g$ has $x_{21}'=x_{21}+kx_{11}$ divisible by $d$ and $x_{11}'=x_{11}$
not divisible by $d$. From $(11)$ we see that this means $g_{2k+1}^{-1}g$ has $x_{31}'$ divisible by $d$ as well.

Now we consider $\Gamma_s^{(d)}$. We claim that
$$\Gamma_s^{(d)}=H^{(d)}\bigcup_{i=1} ^{d}g'_i H^{(d)}$$
where
$$g'_1=\left[\begin{array}{ccc}
0&0&i/\sqrt{d}\\
0&1&0\\
i\sqrt{d}&0&0
\end{array}\right],  g'_{2k}=\left[\begin{array}{ccc}
1&0&ik/\sqrt{d}\\
0&1&0\\
0&0&1
\end{array}\right],  g'_{2k+1}=\left[\begin{array}{ccc}
1&0&-ik/\sqrt{d}\\
0&1&0\\
0&0&1
\end{array}\right]  $$
$(k=1,2,\cdots (d-1)/2)$.

Write a general element of $\Gamma_s^{(d)}$ in the form $(3)$, where $z_{13}=x_{13}/2+iy_{13}/2\sqrt{d}$ and the other entries are as before. In order
to show that an element of $\Gamma_s^{(d)}$ is in $H^{(d)}$, we must show that $d\mid y_{13}$.
If $d\mid y_{13}$, then $g$ is in $H^{(d)}$. If $d\mid x_{33}$, then $g'^{-1}_1g$  has $y'_{13}= -x_{33}$ divisible by $d$.  It is in $H^{(d)}$.

If $d\nmid y_{13}$ and $d\nmid x_{33}$. Then there are also some cases.

Case 1:  If $y_{13}\equiv k x_{33}\pmod 7$, then $g'^{-1}_{2k}g$ has $y'_{13}=y_{13}-kx_{33}$ divisible by $d$.

Case 2:  If $y_{13}\equiv -kx_{33} \pmod 7$, then $g'^{-1}_{2k+1}g$ has $y'_{13}=y_{13}+kx_{33}$ divisible by $d$.

\end{proof}

\begin{proposition}The sister groups $\Gamma^{(d)}_s$ has only one cusp.
\end{proposition}
\begin{proof}
We only give a proof for the case $d=2$. The other cases are similar. Note that  $\Gamma^{(d)}$ has only one cusp.
Without loss of generality we can assume that $D$ is the fundamental domain for $\Gamma^{(2)}$ with the cusp based at $0$. Then the fundamental domain for 
$H^{(2)}$ is $D\bigcup g_1(D)\bigcup g_2(D)$  and  the  candidate cusps of $H^{(2)}$ are $\infty, g_1(0)=g_2(0)=0$. In order to show that $\Gamma^{(2)}_s$ has only one cusp, it is sufficient to show that  there is a transformations $h\in\Gamma^{(2)}_s$ satisfying  $h(0)=\infty$.
In fact, we can find  such $h\in\Gamma^{(2)}_s$, where
$$h=\left[\begin{array}{ccc}
0&0&i/\sqrt{2}\\
0&1&0\\
i\sqrt{2}&0&0
\end{array}\right].$$

\end{proof}

\section{On the structure of the stabilizer}

In this section we will obtain the generators and relations of the stabilizer of the sister of Picard modular groups by analysis of the fundamental domain in Heisenberg group.

\subsection{The stabilizer of $q_\infty$}

First we want to analyse $(\Gamma^{(d)}_s)_\infty$ with $d=2,7,11$, the stabilizer of $q_\infty$. Every element of $(\Gamma^{(d)}_s)_\infty$ is upper triangular and its diagonal entries are units in $\mathcal{O}_d$. Recall that the units of $\mathcal{O}_1$ are $\pm1,\pm i$, they are $\pm1,\pm\omega,\pm\omega^2$ for $\mathcal{O}_3$ and they are $\pm1$ for others. Therefore $(\Gamma^{(d)}_s)_\infty$ contains no dilations and so is a subgroup of $Isom(\mathfrak{R})$ and fits into the exact sequence as
$$0\longrightarrow\mathbb{R}\cap(\Gamma^{(d)}_s)_\infty\longrightarrow
(\Gamma^{(d)}_s)_\infty\stackrel{\Pi_{*}}{\longrightarrow}\Pi_{*}\left((\Gamma^{(d)}_s)_\infty\right)\longrightarrow1.$$

We can write the isometry group of the integer lattice as
\begin{equation*}
Isom(\mathcal{O}_d)=\left\{\left[\begin{array}{cc}\alpha&\beta\\0&1\end{array}\right]:\alpha,\beta\in\mathcal{O}_d, \alpha\ \text{is a unit}\right\}.
\end{equation*}

We now find the image and kernel in this exact sequence.

\begin{proposition}\label{eq:3-1}The stabilizer
$(\Gamma^{(d)}_s)_\infty$ of $q_{\infty}$ in $\Gamma^{(d)}_s$ satisfies
$$0\longrightarrow \frac{2}{\sqrt{d}}\mathbb{Z}\longrightarrow(\Gamma^{(d)}_s)_\infty\stackrel{\Pi_{*}}{\longrightarrow}\Delta^{(d)}\longrightarrow1,$$
where $\Delta^{(d)}\subset Isom(\mathcal{O}_d)$ is of index 2 if $d\equiv2\pmod4$ and $\Delta^{(d)}=Isom(\mathcal{O}_d)$ if $d\equiv3\pmod4$.
\end{proposition}

\begin{proof} From the explicit construction of $\Pi_{\ast}$, we see that for $A\in(\Gamma^{(d)}_s)_\infty$, $\Pi_{\ast}(A)$ is not dependent on the entry $z_{13}$ of $A$. As in the proof
of \cite{z2}, we obtain that $\Delta^{(d)}=\Pi_{\ast}\left((\Gamma^{(d)}_s)_\infty\right)$ is $Isom(\mathcal{O}_d)$ if $d\equiv3(mod\ 4)$ and $\Delta^{(d)}\subset Isom(\mathcal{O}_d)$ is of index 2 if $d\equiv2(mod\ 4)$.
Likewise, the kernel of $\Pi_{\ast}$ is easily seen to consist of those vertical translation in $(\Gamma_d)_\infty$, that is, Heisenberg translation $(0,2n/\sqrt{d})\in\mathfrak{R}$ for $n\in\mathbb{Z}$.
 \end{proof}

\subsection{Generators  for the $\Delta^{(d)}$}
The  fundamental domain and the generators of $\Delta^{(d)}=\Pi_*\left(\Gamma^{(d)}_s)_\infty\right)$ was described explicitly in \cite{z2}.

  A fundamental domain for $\Delta^{(2)}=\Pi_*\left((\Gamma^{(2)}_s)_\infty\right)$ is the triangle in  $\mathbb{C}$ with vertices at $-1+\sqrt{2}i/2$ and $1\pm\sqrt{2}i/2$; see (a) in Figure 3.1.
Side paring maps are given by
$$r^{(2)}_1=\left[\begin{array}{cc}-1&0\\0&1\end{array}\right],\ r^{(2)}_2=\left[\begin{array}{cc}-1&2\\0&1\end{array}\right],\
r^{(2)}_3=\left[\begin{array}{cc}-1&\sqrt{2}i\\0&1\end{array}\right].
$$
The first of these is a rotation of order 2 fixing origin, the second is a rotation of order 2 fixing $1/2$ and the third is a rotation of order 2 fixing $\sqrt{2}i/2$.  One can see that these side paring
maps are also the generators of $\Pi_*\left((\Gamma^{(2)}_s)_\infty\right)$.

 A fundamental domain for $\Delta^{(d)}=Isom(\mathcal{O}_d)$ with $d=7$ or $11$ is the triangle in  $\mathbb{C}$ with vertices at $(-1+i\sqrt{d})/4$, $(1-i\sqrt{d})/4$ and $(3+i\sqrt{d})/4$; see (b) in Figure 3.1. Side paring maps are given by
$$r^{(d)}_1=\left[\begin{array}{cc}-1&0\\0&1\end{array}\right],\ r^{(d)}_2=\left[\begin{array}{cc}-1&1\\0&1\end{array}\right],\
r^{(d)}_3=\left[\begin{array}{cc}-1&(1+i\sqrt{d})/2\\0&1\end{array}\right].
$$
All these maps are rotations by $\pi$ fixing $0,1/2$ and $(1+i\sqrt{d})/4$ respectively.

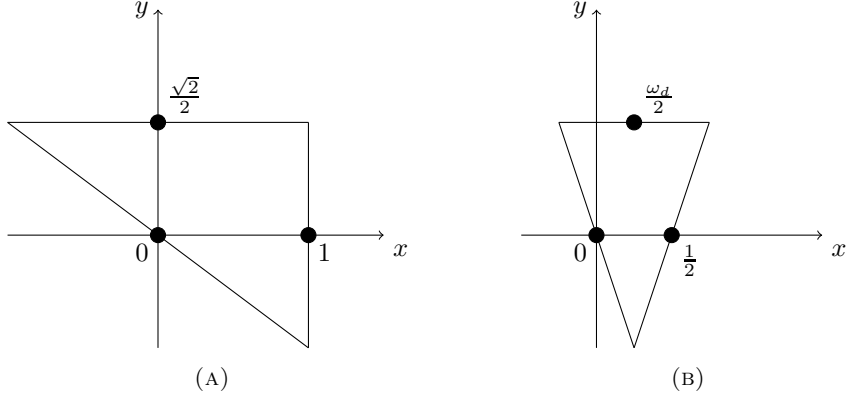
\begin{figure}
\begin{subfigure}{.5 \linewidth}
\centering
\begin{tikzpicture}[scale=1]
\draw[->](-2,0)--(3,0);
\draw[->](0,-1.5,0)--(0,3);
\draw (-2,1.5)--(2,1.5);
\draw(-2,1.5)--(2,-1.5);
\draw(2,-1.5)--(2,1.5);
\draw [fill] (0,0) circle [radius=0.1];
\draw [fill] (0,1.5) circle [radius=0.1];
\draw [fill] (2,0) circle [radius=0.1];
\node [below right] at (3,0) {$x$};
\node [left] at (0,3) {$y$};
\node[below left] at (0,0) {$0$};
\node[below right] at (2,0) {$1$};
\node[above right] at (0,1.5) {$\frac{\sqrt{2}}{2}$};
\end{tikzpicture}
\caption{} 
\end{subfigure}%
\begin{subfigure}{.5\linewidth}
\centering
\begin{tikzpicture}[scale=1]
\draw[->](-1,0)--(3,0);
\draw[->](0,-1.5,0)--(0,3);
\draw (-0.5,1.5)--(1.5,1.5);
\draw(-0.5,1.5)--(0.5,-1.5);
\draw(0.5,-1.5)--(1.5,1.5);
\draw [fill] (0,0) circle [radius=0.1];
\draw [fill] (0.5,1.5) circle [radius=0.1];
\draw [fill] (1,0) circle [radius=0.1];
\node [below right] at (3,0) {$x$};
\node [left] at (0,3) {$y$};
\node[below left] at (0,0) {$0$};
\node[below right] at (1,0) {$\frac{1}{2}$};
\node[above right] at (0.5,1.5) {$\frac{\omega_d}{2}$};
\end{tikzpicture}
\caption{} 
\end{subfigure}%
\caption{(A) Fundamental domian for $\Delta^{(d)}$ with $d\equiv2\pmod4$.  (B). Fundamental domian for $\Delta^{(d)}$ with $d\equiv3\pmod4$.} 
\end{figure}

\medskip

In order to produce a fundamental domain for $(\Gamma^{(d)}_s)_\infty$ we look at all the preimages of
the triangle (that is a fundamental domain of $\Pi_*\left((\Gamma^{(d)}_s)_\infty\right) $ under vertical projection
$\Pi$ and we intersect this with a fundamental domain for $ker(\Pi_*)$. The inverse of the image of the
triangle under $\Pi$ is an infinite prism. The kernel of $\Pi_*$ is the infinite cyclic group
generated by $T^{(d)}$, the vertical translation by $(0,2/\sqrt{d})$.

\medskip

Next, we will study the fundamental domain and the generators of stabilizer case by case.

\begin{proposition} $(\Gamma^{(2)}_s)_\infty$ is generated by $$ R^{(2)}_1=\left[\begin{array}{ccc}1&0&0\\0&-1&0\\0&0&1\end{array}\right],
 R^{(2)}_2=\left[\begin{array}{ccc}1&2&-2\\0&-1&2\\0&0&1\end{array}\right],
 R^{(2)}_3=\left[\begin{array}{ccc}1&-i\sqrt{2}&-1\\0&-1&i\sqrt{2}\\0&0&1\end{array}\right]$$
and $$T^{(2)}=\left[\begin{array}{ccc}1&0&i/\sqrt{2}\\0&1&0\\0&0&1\end{array}\right].$$

A presentation is given by
\begin{align*}(\Gamma^{(2)}_s)_{\infty}&=\langle R^{(2)}_j,T^{(2)}|T^{(2)}R^{(2)}_1{T^{(2)}}^{-1}R^{(2)}_1={R^{(2)}_1}^2={T^{(2)}}^{-1}R^{(2)}_2T^{(2)}R^{(2)}_2={R^{(2)}_2}^2\\
&=({R^{(2)}_1}^{-1}{T^{(2)}}^{-2}{R^{(2)}_3}^{-1}{T^{(2)}}^{-2}{R^{(2)}_1)}^2=T^{(2)}{R^{(2)}_3}^{-1}{T^{(2)}}^{-1}R^{(2)}_3\rangle.
\end{align*}
\end{proposition}

\begin{proof}
Those matrices are constructed by lifting generators of the subgroup $\Delta^{(2)}\subset Isom(\mathcal{O}_2)$ of index 2
and also $T^{(2)}$ is a generator of the kernel of the map $\Pi_*$. A fundamental domain can be constructed with
side pairings as Figure 2, where the vertices of the prism in horospherical coordinate are
$v^{+}_1=(1-i\sqrt{2}/2,\sqrt{2}/2), v^{+}_2=(1+i\sqrt{2}/2,\sqrt{2}/2), v^{+}_3=(-1+i\sqrt{2}/2,\sqrt{2}/2)$,  for the upper cap of the prism
and $v^{-}_1=(1-i\sqrt{2}/2,-\sqrt{2}/2), v^{-}_2=(1+i\sqrt{2}/2,-\sqrt{2}/2), v^{-}_3=(-1+i\sqrt{2}/2,-\sqrt{2}/2)$ for the base. 
In order to get the side-pairing maps, we need to introduce some  points
 on the faces of the prism. See Table 1.

\begin{table}[ht]
\caption{Points introduced on the prism $\Sigma_2$} 
\centering 
\begin{tabular}{c c c c} 
\hline 
Points & Horospherical coordinate & Points & Horospherical coordinate \\ [0.5ex] 
\hline 
$u^{\pm}_1$&$(0,\pm\sqrt{2}/2)$&$u^{\pm}_2$&$(1-i\sqrt{2}/4\pm\sqrt{2}/2)$\\ 
$u^{\pm}_3$&$(1,\pm\sqrt{2}/2)$&$u^{\pm}_4$&$(1+i\sqrt{2}/4,\pm\sqrt{2}/2)$\\
$u^{\pm}_5$&$(1/2+i\sqrt{2}/2,\pm\sqrt{2}/2)$&$u^{\pm}_6$&$(i\sqrt{2}/2,\pm\sqrt{2}/2) $\\
$u^{\pm}_7$&$(-1/2+i\sqrt{2}/2,\pm\sqrt{2}/2)$&& \\
\hline 
\end{tabular}
\label{table:nonlin}
\end{table}

The actions of the side-pairing maps on $\mathfrak{R}$ are given by
\begin{align*}
R^{(2)}_1(z,t)&=(-z,t),\\
R^{(2)}_2(z,t)&=(-z+2,t+4\Im m{z}),\\
R^{(2)}_3(z,t)&=(-z+i\sqrt{2},t-2\sqrt{2}\Re z),\\
T^{(2)}(\zeta,t)&=(\zeta,t+\sqrt{2}).
\end{align*}
We describe the side pairings in terms of their actions on the vertices:

\begin{eqnarray*}
R^{(2)}_1&: &(u^{+}_1, u^{-}_1, v^{+}_1, v^{-}_1)\rightarrow (u^{+}_1, u^{-}_1, v^{+}_3, v^{-}_3)\\
R^{(2)}_2&: &(u^{+}_2, u^{+}_3, u^{-}_3)\rightarrow (u^{-}_4, u^{+}_3, u^{-}_3)\\
T^{(2)}R^{(2)}_2&: &(v_1^{+}, u_2^{+}, u_3^{-}, u_2^{-})\rightarrow (v_2^{-}, u_4^{+}, u_3^{+}, u_4^{-})\\
{T^{(2)}}^2R^{(2)}_2&: &(v_1^{+}, v_1^{-}, u_2^{-})\rightarrow (v_2^{+}, v_2^{-}, u_4^{+})\\
R^{(2)}_3&: &(u_5^{+}, u_6^{+}, u_6^{-})\rightarrow (u_7^{-}, u_6^{+}, u_6^{-})\\
T^{(2)}R^{(2)}_3&: &(u_5^{+}, u_6^{-}, u_5^{-}, v_2^{+})\rightarrow (u_7^{+}, u_6^{+}, u_7^{-}, v_3^{-})\\
{T^{(2)}}^2R^{(2)}_3&:& (v_2^{+}, v_2^{-}, u_5^{-})\rightarrow (v_3^{+}, v_3^{-}, u_7^{+})\\
T^{(2)}&: &(v_1^{-}, v_2^{-}, v_3^{-}, u_1^{-}, u_2^{-}, u_3^{-}, u_4^{-}, u_5^{-}, u_6^{-}, u_7^{-}) \rightarrow \\ &&(v_1^{+}, v_2^{+}, v_3^{+}, u_1^{+}, u_2^{+}, u_3^{+}, u_4^{+}, u_5^{+}, u_6^{+}, u_7^{+} )
\end{eqnarray*}

The presentation can be obtained  from the edge cycles of the fundamental domain.
\end{proof}

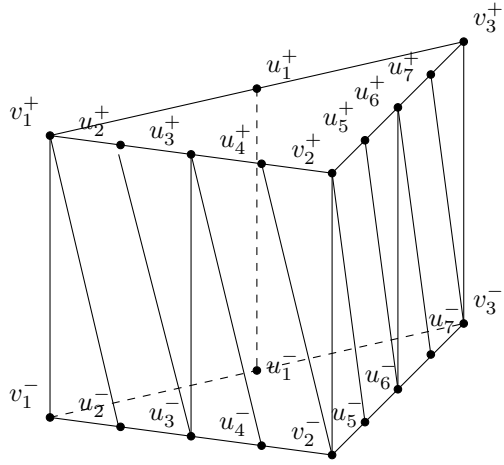
\begin{figure}
\begin{tikzpicture}[x=.5cm,y=.5cm]
\draw (0,0)--(11,2.5)--(7.5,-1)--cycle;
\draw (0,-7.5)--(7.5,-8.5)--(11,-5);
\draw[style=dashed] (0,-7.5) -- (11,-5);   
\draw (0,-7.5)--(0,0);  
\draw (11,-5)--(11,2.5);  
  \draw (7.5,-8.5)--(7.5,-1);    
\draw [fill] (0,0) circle [radius=0.1];
 \draw [fill] (0,-7.5) circle [radius=0.1];   
\draw [fill] (11,2.5) circle [radius=0.1];  \draw [fill] (11,-5) circle [radius=0.1]; 
 \draw [fill] (7.5,-1) circle [radius=0.1];   \draw [fill] (7.5,-8.5) circle [radius=0.1];
\draw [fill] (5.5,1.25) circle [radius=0.1];    \draw [fill] (5.5,-6.25) circle [radius=0.1]; \draw[style=dashed] (5.5,1.25) -- (5.5,-6.25); 
\draw [fill] (1.875,-0.25) circle [radius=0.1];  \draw [fill] (1.875,-7.75) circle [radius=0.1]; 
\draw [fill] (3.75,-0.5) circle [radius=0.1];  \draw [fill] (3.75,-8) circle [radius=0.1]; 
 \draw [fill] (5.625,-0.75) circle [radius=0.1];  \draw [fill] (5.625,-8.25) circle [radius=0.1];    
 
 \draw [fill] (8.375,-0.125) circle [radius=0.1];   \draw [fill] (8.375,-7.625) circle [radius=0.1];
  \draw [fill] (9.25,0.75) circle [radius=0.1];   \draw [fill] (9.25,-6.75) circle [radius=0.1];   
    \draw [fill] (10.125,1.625) circle [radius=0.1];      \draw [fill] (10.125,-5.825) circle [radius=0.1];  
 
 \draw (0,0)--(1.825,-7.75) ; \draw(1.825,-0.5)--(3.75,-8);\draw(3.75,-8)--(3.75,-0.5);\draw(3.75,-0.5)--(5.625,-8.25);\draw(5.625,-0.75)--(7.5,-8.5);
 \draw(7.5,-1)--(8.375,-7.625);\draw(8.375,-0.125)--(9.25,-6.75);\draw(9.25,0.75)--(9.25,-6.75);\draw(9.25,0.75)--(10.125,-5.825);\draw(10.125,1.625)--(11,-5);

 \node[above left] at (0,0) {$v^{+}_1$};   \node[above left] at (0,-7.5) {$v^{-}_1$};  
  \node[above left] at (7.5,-1) {$v^{+}_2$};   \node[above left] at (7.5,-8.5) {$v^{-}_2$};   
 \node[above right] at (11,2.5) {$v^{+}_3$};     \node[above right] at (11,-5) {$v^{-}_3$};    
  \node[above right] at (5.5,1.25) {$u^{+}_1$};  \node[above right] at (5.5,-6.75) {$u^{-}_1$};   
  
  \node[above left] at (1.875,-0.25) {$u^{+}_2$};    \node[above left] at (1.875,-7.75) {$u^{-}_2$};      
 \node[above left] at (3.75,-0.5) {$u^{+}_3$};  \node[above left] at (3.75,-8) {$u^{-}_3$}; 
    \node[above left] at (5.625,-0.75) {$u^{+}_4$};   \node[above left] at (5.625,-8.25) {$u^{-}_4$};  
     
    \node[above left ] at (8.375,-0.15) {$u^{+}_5$};    \node[above left ] at (8.6,-8) {$u^{-}_5$};      
 \node[above left] at (9.2,0.55) {$u^{+}_6$};  \node[above left] at (9.5,-7) {$u^{-}_6$}; 
    \node[above left] at (10.1,1.2) {$u^{+}_7$};   \node[above ] at (10.5,-5.5) {$u^{-}_7$};    
  \end{tikzpicture}  
  \caption{Fundamental domain for $(\Gamma_{s}^{(2)})_{\infty}$ in Heisenberg group}
  
  \end{figure}

\medskip

\begin{proposition} $(\Gamma^{(7)}_s)_\infty$ is generated by $$ R^{(7)}_1=\left[\begin{array}{ccc}1&0&0\\0&-1&0\\0&0&1\end{array}\right],
 R^{(7)}_2=\left[\begin{array}{ccc}1&1&i\omega_7/\sqrt{7}\\0&-1&1\\0&0&1\end{array}\right],
 R^{(7)}_3=\left[\begin{array}{ccc}1&\bar{\omega}_7&-1\\0&-1&\omega_7\\0&0&1\end{array}\right]$$
and $$T^{(7)}=\left[\begin{array}{ccc}1&0&i/\sqrt{7}\\0&1&0\\0&0&1\end{array}\right].$$

A presentation is given by
\begin{align*}(\Gamma_7)_{\infty}&=\langle R^{(7)}_j,T^{(7)}|T^{(7)}R^{(7)}_1{T^{(7)}}^{-1}R^{(7)}_1={R^{(7)}_1}^2=T^{(7)}{R^{(7)}_2}^{-1}{T^{(7)}}^{-1}R^{(7)}_2\\
&=(R^{(7)}_1T^{(7)}R^{(7)}_3{T^{(7)}}^2R^{(7)}_2)^2=(R^{(7)}_1{T^{(7)}}^2R^{(7)}_3T^{(7)}R^{(7)}_2)^2\\ &={R^{(7)}_2}^{-1}T^{(7)}{R^{(7)}_2}^{-1}=T^{(7)}{R^{(7)}_3}^{-1}{T^{(7)}}^{-1}R^{(7)}_3={R^{(7)}_3}^2\rangle.
\end{align*}
\end{proposition}

\begin{proof}
Those matrices are constructed by lifting generators of $Isom(\mathcal{O}_7)$ and also $T^{(7)}$ is a generator of the kernel of the map $\Pi_*$.
A fundamental domain can be constructed with
side pairings as Figure 3, where the vertices of the prism are
$v^{+}_1=(1/4-i\sqrt{7}/4,\sqrt{7}/7), v^{+}_2=(3/4+i\sqrt{7}/4,\sqrt{7}/7),v^{+}_3=(-1/4+i\sqrt{7}/4,\sqrt{7}/7)$ for the upper cap of the prism
and $v^{-}_1=(1/4-i\sqrt{7}/4,-\sqrt{7}/7), v^{-}_2=(3/4+i\sqrt{7}/4,-\sqrt{7}/7),v^{-}_3=(-1/4+i\sqrt{7}/4,-\sqrt{7}/7),$   for the base.
We also  introduce more points  on the faces of the prism which allow us to get the side-pairing maps.  See Table 2.

\begin{table}[ht]
\caption{Points introduced on the prism $\Sigma_7$.} 
\centering 
\begin{tabular}{c c c c} 
\hline 
Points & Horospherical coordinate & Points & Horospherical coordinate \\ [0.5ex] 
\hline 
$u^{\pm}_1$&$(0,\pm\sqrt{7}/7)$&$u^{\pm}_2$&$(2/7-i3\sqrt{7}/14, \pm\sqrt{7}/7)$\\ 
$u^{\pm}_3$&$(3/7-i\sqrt{7}/14,\pm\sqrt{7}/7)$&$u^{\pm}_4$&$(2/7-i3\sqrt{7}/14,\pm\sqrt{7}/7)$\\
$u^{\pm}_5$&$(4/7-i\sqrt{7}/14,\pm\sqrt{7}/7)$&$u^{\pm}_6$&$(5/7+i3\sqrt{7}/14,\pm\sqrt{7}/7)$\\
$u^{\pm}_7$&$(15/28+i\sqrt{7}/4,\pm\sqrt{7}/7)$&$u^{\pm}_8$&$(1/4+i\sqrt{7}/4,\pm\sqrt{7}/7) $ \\
$u^{\pm}_9$&$(-1/28+i\sqrt{7}/4,\pm\sqrt{7}/7)$&$w_1$&$(1/4-i\sqrt{7}/4,-\sqrt{7}/14)$\\
$w_2$&$(1/2,0)$&$w_3$&$(3/4+i\sqrt{7}/4,\sqrt{7}/14)$\\
$w_4$&$(-1/4+i\sqrt{7}/4,-\sqrt{7}/14)$&&\\
\hline 
\end{tabular}
\label{table:nonlin}
\end{table}

The actions of side-pairing maps on $\mathfrak{R}$ are given by

\begin{align*}
R^{(7)}_1(z,t)&=(-z,t),\\
R^{(7)}_2(z,t)&=\left(-z+1,t+2\Im(z)+\frac{1}{\sqrt{7}}\right),\\
R^{(7)}_3(z,t)&=\left(-z+\omega_7,t-2\Im(\overline{\omega_7}z)\right),\\
T^{(7)}(z,t)&=(z,t+2/\sqrt{7})
\end{align*}

We describe the side pairings in terms of their actions on the vertices:

\begin{eqnarray*}
R^{(7)}_1&: &(u^{+}_1,u^{-}_1,v^{+}_1,v^{-}_1)\rightarrow (u^{+}_1,u^{-}_1,v^{+}_3, v^{-}_3),\\
R^{(7)}_2&: &(u^{+}_2,u^{-}_3,u^{-}_4,w_2,u_3^{+})\rightarrow (u^{-}_6,u^{-}_5,w_2,u^{+}_4,u_5^{+}),\\
T^{(7)}R^{(7)}_2&: &(w_1,u_2^{-},u_3^{-},u_2^{+},v_1^{+})\rightarrow (v_2^{-},u_6^{-},u_5^{+},u_6^{+},w_3),\\
{T^{(7)}}^2R^{(7)}_2&: &(w_1,v_1^{-},u_2^{-})\rightarrow (v_2^{+},w_3,u_6^{+}),\\
{T^{(7)}}^{-1}R^{(7)}_2&:& (u_3^{+},w_2,u_4^{+})\rightarrow (u_5^{-},u_4^{-},w_2),\\
{T^{(7)}}^2R^{(7)}_3&: &(w_3,v_2^{-},u_7^{-})\rightarrow (v_3^{+},w_4,u_9^{+}),\\
T^{(7)}R^{(7)}_3&:&(v_2^{+}w_3u_7^{-}u_7^{+}u_8^{-})\rightarrow (w_4v_3^{-}u_9^{-}u_9^{+}u_8^{+}),\\
R^{(7)}_3&:& (u_8^{-},u_8^{+},u_7^{+})\rightarrow (u_8^{-},u_8^{+},u_9^{-}),\\
T^{(7)}&:&(v_1^{-},u^{-}_2,u^{-}_3,u^{-}_4,u^{-}_5,u^{-}_6,v_2^{-}, u^{-}_7,u^{-}_8, u^{-}_9, v_3^{-})\rightarrow   \\   &&(v_1^{+},u^{+}_2,u^{+}_3,u^{+}_4,u^{+}_5,u^{+}_6,v_2^{+}, u^{+}_7,u^{+}_8, u^{+}_9, v_3^{+}).
\end{eqnarray*}

The presentation can be obtained  from the edge cycles of the fundamental domain.

\end{proof}

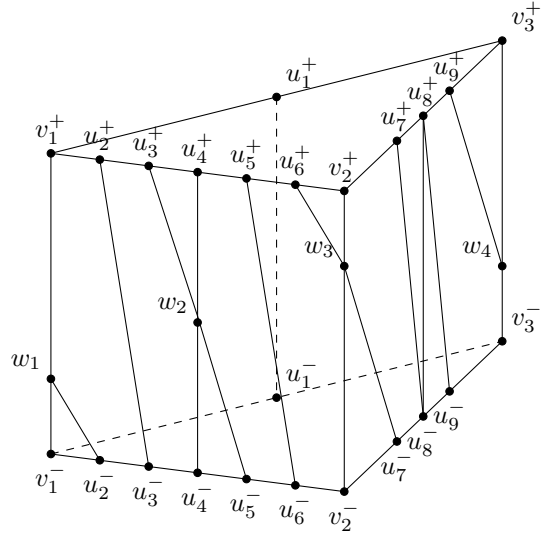
\begin{figure}
\begin{tikzpicture}[x=.5cm,y=.5cm]
\draw (0,0)--(12,3)--(7.8,-1)--cycle;
\draw (0,-8)--(7.8,-9)--(12,-5);
\draw[style=dashed] (0,-8) -- (12,-5);   
\draw (0,-8)--(0,0);  
\draw (12,-5)--(12,3);  
  \draw (7.8,-1)--(7.8,-9);    
\draw [fill] (0,0) circle [radius=0.1];
 \draw [fill] (0,-8) circle [radius=0.1];   
\draw [fill] (12,3) circle [radius=0.1];  \draw [fill] (12,-5) circle [radius=0.1]; 
 \draw [fill] (7.8,-1) circle [radius=0.1];   \draw [fill] (7.8,-9) circle [radius=0.1];
\draw [fill] (6,1.5) circle [radius=0.1];    \draw [fill] (6,-6.5) circle [radius=0.1]; \draw[style=dashed] (6,1.5) -- (6,-6.5); 
 
 \node[above ] at (0,0) {$v^{+}_1$};   \node[below] at (0,-8) {$v^{-}_1$};  
  \node[above ] at (7.8,-1) {$v^{+}_2$};   \node[below] at (7.8,-9) {$v^{-}_2$};   
 \node[above right] at (12,3) {$v^{+}_3$};     \node[above right] at (12,-5) {$v^{-}_3$};    
  \node[above right] at (6,1.5) {$u^{+}_1$};  \node[above right] at (6,-6.5) {$u^{-}_1$};  
  
\draw [fill] (1.3,-1/6) circle [radius=0.1];  \node[above ] at (1.3,-1/6) {$u^{+}_2$};  \draw [fill] (1.3,-1/6-8) circle [radius=0.1];\node[below] at (1.3,-1/6-8) {$u^{-}_2$};
  \draw [fill] (2.6,-2/6) circle [radius=0.1]; \node[above ] at (2.6,-2/6) {$u^{+}_3$}; \draw [fill] (2.6,-2/6-8) circle [radius=0.1];\node[below ] at (2.6,-2/6-8) {$u^{-}_3$}; 
    \draw [fill] (3.9,-3/6) circle [radius=0.1];  \node[above ] at (3.9,-3/6) {$u^{+}_4$};   \draw [fill] (3.9,-3/6-8) circle [radius=0.1]; \node[below ] at (3.9,-3/6-8) {$u^{-}_4$}; \draw [fill] (3.9,-3/6-4) circle [radius=0.1]; 
    \node[above  left] at (0,-6) {$w_1$};  \draw [fill] (5.2,-4/6) circle [radius=0.1];   \node[above ] at (5.2,-4/6) {$u^{+}_5$};    \draw [fill] (5.2,-4/6-8) circle [radius=0.1];     \node[below ] at (5.2,-4/6-8) {$u^{-}_5$};    
\draw [fill] (6.5,-5/6) circle [radius=0.1];     \node[above ] at (6.5,-5/6) {$u^{+}_6$};     \draw [fill] (6.5,-5/6-8) circle [radius=0.1];   \node[below ] at (6.5,-5/6-8) {$u^{-}_6$};        
  \draw [fill] (0,-6) circle [radius=0.1];\draw [fill] (7.8,-3) circle [radius=0.1]; \node[above left ] at (3.9,-3/6-4) {$w_2$};\node[above  left] at (7.8,-3) {$w_3$};

 \draw (0,-6)--(1.3,-1/6-8);\draw(6.5,-5/6)--(7.8,-3); \draw (1.3,-1/6)--(2.6,-2/6-8);\draw(3.9,-3/6)--(3.9,-3/6-8);
  \draw(2.6,-2/6)--(3.9,-2/6-4);\draw(3.9,-2/6-4)--(5.2,-4/6-8);\draw(5.2,-4/6)--(6.5,-5/6-8);

\draw [fill] (7.8+2*0.7,-1+4/3) circle [radius=0.1];      \node[above ] at (7.8+0.7*2,-1+4/3) {$u^{+}_7$};           \draw [fill] (7.8+2*0.7,-1+4/3-8) circle [radius=0.1];  \node[below ] at (7.8+0.7*2,-1+4/3-8) {$u^{-}_7$};   
 \draw [fill] (7.8+3*0.7,-1+6/3) circle [radius=0.1];    \node[above ] at (7.8+0.7*3,-1+6/3) {$u^{+}_8$};       \draw [fill] (7.8+3*0.7,-1+6/3-8) circle [radius=0.1];     \node[below ] at (7.8+0.7*3,-1+6/3-8) {$u^{-}_8$};   
  \draw [fill] (7.8+4*0.7,-1+8/3) circle [radius=0.1];     \node[above ] at (7.8+0.7*4,-1+8/3) {$u^{+}_9$};          \draw [fill] (7.8+4*0.7,-1+8/3-8) circle [radius=0.1];    \node[below ] at (7.8+0.7*4,-1+8/3-8) {$u^{-}_9$};
  
     \draw [fill] (12,-3) circle [radius=0.1];    \node[above left ] at (12,-3) {$w_4$};  
\draw(7.8,-3)--(7.8+0.7*2,-1+4/3-8);     

\draw(7.8+0.7*2,-1+4/3)--(7.8+0.7*3,-1+2-8);\draw(7.8+0.7*3,-1+6/3)--(7.8+0.7*3,-1+2-8);     \draw(7.8+0.7*3,-1+6/3)--(7.8+0.7*4,-1+8/3-8);   
 \draw(7.8+0.7*4,-1+8/3)--(12,3-6);

    \end{tikzpicture}  
  \caption{Fundamental domain for $(\Gamma_{s}^{(7)})_{\infty}$ in Heisenberg group}
\end{figure}

\medskip

\begin{proposition} $(\Gamma^{(11)}_s)_\infty$ is generated by $$ R^{(11)}_1=\left[\begin{array}{ccc}1&0&0\\0&-1&0\\0&0&1\end{array}\right],
 R^{(11)}_2=\left[\begin{array}{ccc}1&1&i\omega_{11}/\sqrt{11}\\0&-1&1\\0&0&1\end{array}\right],
 $$
$$R^{(11)}_3=\left[\begin{array}{ccc}1&\bar{\omega}_{11}&i(-1+3\omega_{11})/\sqrt{11}\\0&-1&\omega_{11}\\0&0&1\end{array}\right]\quad \text{and}\quad T^{(11)}=\left[\begin{array}{ccc}1&0&i/\sqrt{11}\\0&1&0\\0&0&1\end{array}\right].$$

A presentation is given by
\begin{align*}(\Gamma_{11})_{\infty}&=\langle R^{(11)}_j,T^{(11)}|T^{(11)}{R^{(11)}_3}^{-1}{T^{(11)}}^{-1}R^{(11)}_3 =T^{(11)}R^{(11)}_1{T^{(11)}}^{-1}R^{(11)}_1\\
&={R^{(11)}_1}^{-1}{T^{(11)}}^3R^{(11)}_3{T^{(11)}}^2R^{(11)}_2{R^{(11)}_1}^{-1}{T^{(11)}}^2R^{(11)}_3{T^{(11)}}^2R^{(11)}_2\\
&={R^{(11)}_1}^{-1}{T^{(11)}}^2R^{(11)}_3{T^{(11)}}^3R^{(11)}_2{R^{(11)}_1}^{-1}{T^{(11)}}^2R^{(11)}_3{T^{(11)}}^2R^{(11)}_2\\&
=T^{(11)}{R^{(11)}_2}^{-1}{T^{(11)}}^{-1}R^{(11)}_2={R^{(11)}_1}^2\rangle.
\end{align*}
\end{proposition}

\begin{proof}
Those matrices are constructed by lifting generators of $Isom(\mathcal{O}_{11})$ and also $T^{(11)}$ is a generator of the kernel of the map $\Pi_*$.
A fundamental domain can be constructed with
side pairings as Figure 4, where the vertices of the prism are
$v^{\pm}_1=(1/4-i\sqrt{11}/4,\pm\sqrt{11}/11), v^{\pm}_2=(3/4+i\sqrt{11}/4,\pm\sqrt{11}/11),v^{\pm}_3=(-1/4+i\sqrt{11}/4,\pm\sqrt{11}/11)$. 
In particular, we introduce more  points on the faces of the prism. See Table 3.

\begin{table}[ht]
\caption{Points introduced on the prism $\Sigma_{11}$} 
\centering 
\begin{tabular}{c c c c} 
\hline 
Points & Horospherical coordinate & Points & Horospherical coordinate \\ [0.5ex] 
\hline 
$u^{\pm}_1$&$(0,\pm\sqrt{11}/11)$&$u^{\pm}_2$&$(3/11-i5\sqrt{11}/22,\pm\sqrt{11}/11)$\\ 
$u^{\pm}_3$&$(8/22-i3\sqrt{11}/22,\pm\sqrt{11}/11)$&$u^{\pm}_4$&$(10/22-i\sqrt{11}/22,\pm\sqrt{11}/11)$\\
$u^{\pm}_5$&$(1/2,\pm\sqrt{11}/11)$&$u^{\pm}_6$&$(12/22+i\sqrt{11}/22,\pm\sqrt{11}/11)$\\
$u^{\pm}_7$&$(14/22+i3\sqrt{11}/22,\pm\sqrt{11}/11)$&$u^{\pm}_8$&$(8/11+i5\sqrt{11}/22,\pm\sqrt{11}/11) $ \\
$u^{\pm}_9$&$(31/44+i\sqrt{11}/44,\pm\sqrt{11}/11)$& $u^{\pm}_{10}$&$(23/44+i\sqrt{11}/4,\pm\sqrt{11}/11)$\\
$ u^{\pm}_{11}$&$(15/44-i\sqrt{11}/4,\pm\sqrt{11}/11)$&$ u^{\pm}_{12}$&$(1/4+i\sqrt{11}/4,\pm\sqrt{11}/11)$\\
$u^{\pm}_{13}$&$(7/44+i\sqrt{11}/4,\pm\sqrt{11}/11)$&$u^{\pm}_{14}$&$(-1/44+i\sqrt{11}/4,\pm\sqrt{11}/11)$\\
$u^{\pm}_{15}$&$(-9/44+i\sqrt{11}/4,\pm\sqrt{11}/11)$&$w_1$&$(1/4-i\sqrt{11}/4, -\sqrt{11}/22)$\\
$w_2$&$(1/2, 0)$&$w_3$&$(3/4+i\sqrt{11}/4, \sqrt{11}/22)$\\
$w_4$&$(3/4+i\sqrt{11}/4, -\sqrt{11}/22)$&$w_5$&$(1/4+i\sqrt{11}/4, 0)$\\
$w_6$&$(-1/4+i\sqrt{11}/4, \sqrt{11}/22)$&$w_7$&$(-1/4+i\sqrt{11}/4, 0)$\\
$w_8$&$(-1/4+i\sqrt{11}/4, -\sqrt{11}/22)$&$w_9$&$(3/4+i\sqrt{11}/4, 0)$\\
$w_{10}$&$(1/4-i\sqrt{11}/4, \sqrt{11}/22)$&$w_{11}$&$(1/4-i\sqrt{11}/4, 0)$\\
\hline 
\end{tabular}
\label{table:nonlin}
\end{table}

The actions of side-pairing maps on $\mathfrak{R}$ are given by
\begin{align*}
R^{(11)}_1(z,t)&=(-z,t),\\
R^{(11)}_2(z,t)&=\left(-z+1,t+2\Im(z)+\frac{1}{\sqrt{11}}\right),\\
R^{(11)}_3(z,t)&=\left(-z+\omega_7,t+2\Im(\overline{\omega_{11}}z)+\frac{1}{\sqrt{11}}\right),\\
T^{(11)}(z,t)&=(z,t+2/\sqrt{11}).
\end{align*}

We describe the side pairing in terms of the action on the vertice:

\begin{eqnarray*}
R^{(11)}_1&:&(u^{+}_1,u^{-}_1,v^{-}_1,w_1,w_{11},w_{10},v^{+}_1)\rightarrow (u^{+}_1,u^{-}_1,v^{-}_3,w_8,w_7,w_6,v^{+}_3),\\
R^{(11)}_2&:&(u^{+}_3,u^{-}_4,u^{+}_4,w_2,u_5^{-})\rightarrow (u^{-}_7,u^{-}_6,u^{+}_6,u_5^{+},w_2),\\
T^{(11)}R^{(11)}_2&:&(u_2^{+},u_3^{-},u_3^{+},u_4^{-})\rightarrow (u_8^{-},u_7^{-},u_7^{+},u_6^{+}),\\
{T^{(11)}}^3R^{(11)}_2&:& (w_1,v_1^{-},u_2^{-})\rightarrow (v_2^{+},w_3,u_8^{+}),\\
{T^{(11)}}^2R^{(11)}_2&:& (w_1,u_2^{-},u_3^{-},u_2^{+},v_1^{+},w_{10},w_{11})\rightarrow (v_2^{-},u_8^{-},u_7^{+},u_8^{+},w_3,w_9,w_4)\\
{T^{(11)}}^{-1}R^{(11)}_2&:& (u_4^{+},w_2,u_5^{+})\rightarrow (u_6^{-},u_5^{-},w_2),\\
{T^{(11)}}^3R^{(11)}_3&:&(w_4,v_2^{-},u_9^{-})\rightarrow (v_3^{+},w_6,u_{15}^{+}),\\
{T^{(11)}}^2R^{(11)}_3&:&(w_3,w_9,w_4,u_9^{-},u_{10}^{-}, u_9^{+},v_{2}^{+})\rightarrow (w_7,w_8,v_3^{-}, u_{15}^{-},u_{14}^{+}, u_{15}^{+} ,w_6),\\
T^{(11)}R^{(11)}_3&:&(u_9^{+},u_{10}^{-},u_{11}^{-}, u_{10}^{+})\rightarrow (u_{15}^{-},u_{14}^{-}, u_{13}^{+}, u_{14}^{+}),\\
R^{(11)}_3&:&(u_{10}^{+},u_{11}^{-}, u_{12}^{-}, w_5, u_{11}^{+})  \rightarrow (u_{14}^{-},u_{13}^{-}, w_5,u_{12}^{+}, u_{13}^{+}),\\
{T^{(11)}}^{-1}R^{(11)}_3&:&(u_{11}^{+},w_5,u_{12}^{+})\rightarrow (u_{13}^{-},u_{12}^{-},w_5),\\
T^{(11)}&:&(v_1^{-},u^{-}_2,u^{-}_3,u^{-}_4,u^{-}_5,u^{-}_6,  u^{-}_7,u^{-}_8,  v_2^{-} ,u^{-}_9,  u_{10}^{-},u_{11}^{-} ,u_{12}^{-},u_{13}^{-}, u_{14}^{-},u_{15}^{-}, v_3^{-}) \rightarrow \\
  &&(v_1^{+},u^{+}_2,u^{+}_3,u^{+}_4,u^{+}_5,u^{+}_6,  u^{+}_7,u^{+}_8, v_2^{+} ,u^{+}_9, u_{10}^{+},u_{11}^{+}, u_{12}^{+},u_{13}^{+}, u_{14}^{+},u_{15}^{+},v_3^{+}).
\end{eqnarray*}

The presentation can be obtained following from the edge cycles of the fundamental domain.

\end{proof}

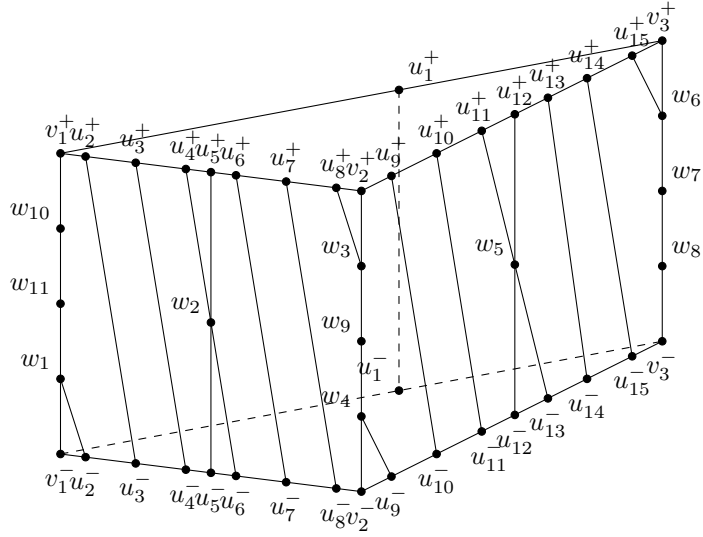
\begin{figure}
\begin{tikzpicture}[x=.5cm,y=.5cm]
\draw (0,0)--(16,3)--(8,-1)--cycle;
\draw (0,-8)--(8,-9)--(16,-5);
\draw[style=dashed] (0,-8) -- (16,-5);   
\draw (0,-8)--(0,0);  
\draw (16,-5)--(16,3);  
  \draw (8,-1)--(8,-9);    
\draw [fill] (0,0) circle [radius=0.1];
 \draw [fill] (0,-8) circle [radius=0.1];   
\draw [fill] (16,3) circle [radius=0.1];  \draw [fill] (16,-5) circle [radius=0.1]; 
 \draw [fill] (8,-1) circle [radius=0.1];   \draw [fill] (8,-9) circle [radius=0.1];
\draw [fill] (9,27/16) circle [radius=0.1];    \draw [fill] (9,27/16-8) circle [radius=0.1]; \draw[style=dashed] (9,27/16) -- (9,27/16-8); 
 
 \node[above ] at (0,0) {$v^{+}_1$};   \node[below] at (0,-8) {$v^{-}_1$};  
  \node[above ] at (8,-1) {$v^{+}_2$};   \node[below] at (8,-9) {$v^{-}_2$};   
 \node[above ] at (16,3) {$v^{+}_3$};     \node[below] at (16,-5) {$v^{-}_3$};    
  \node[above right] at (9,27/16) {$u^{+}_1$};  \node[above left] at (9,27/16-8) {$u^{-}_1$};  
  
\draw [fill] (8/12,-1/12) circle [radius=0.1];  \node[above ] at (8/12,-1/12) {$u^{+}_2$};  \draw [fill] (8/12,-1/12-8) circle [radius=0.1];\node[below] at (8/12,-1/12-8) {$u^{-}_2$};

  \draw [fill] (2,-2/8) circle [radius=0.1]; \node[above ] at (2,-2/8) {$u^{+}_3$}; \draw [fill] (2,-2/8-8) circle [radius=0.1]; \node[below ] at (2,-2/8-8) {$u^{-}_3$}; 
  
  \draw [fill] (40/12,-5/12) circle [radius=0.1];\node[above ] at (40/12,-5/12) {$u^{+}_4$};  \draw [fill] (40/12,-5/12-8) circle [radius=0.1];\node[below ] at (40/12,-5/12-8) {$u^{-}_4$}; 
  
      \draw [fill] (4,-3/6) circle [radius=0.1];  \node[above ] at (4,-3/6) {$u^{+}_5$};   \draw [fill] (4,-3/6-8) circle [radius=0.1]; \node[below ] at (4,-3/6-8) {$u^{-}_5$}; 
      
      \draw [fill] (4,-3/6-4) circle [radius=0.1]; \node[above  left] at (4,-3/6-4) {$w_2$}; 
      
  \draw [fill] (0,-6) circle [radius=0.1];          \node[above  left] at (0,-6) {$w_1$};  \draw [fill] (56/12,-7/12) circle [radius=0.1];   \node[above ] at (56/12,-7/12) {$u^{+}_6$};    \draw [fill] (56/12,-7/12-8) circle [radius=0.1];     \node[below ] at (56/12,-7/12-8) {$u^{-}_6$}; 
    
       \draw [fill] (72/12,-9/12) circle [radius=0.1];     \node[above ] at (72/12,-9/12) {$u^{+}_7$};     \draw [fill] (6,-9/12-8) circle [radius=0.1];   \node[below ] at (6,-9/12-8) {$u^{-}_7$};  

\draw [fill] (88/12,-11/12) circle [radius=0.1];     \node[above ] at (88/12,-11/12) {$u^{+}_8$};     \draw [fill] (88/12,-11/12-8) circle [radius=0.1];   \node[below ] at (88/12,-11/12-8) {$u^{-}_8$};  
      
  \draw [fill] (8,-3) circle [radius=0.1];\node[above  left] at (8,-3) {$w_3$};  
  
     \draw (0,-6)--(8/12,-1/12-8);\draw(8/12,-1/12)--(2,-2/8-8); \draw (40/12,-5/12)--(56/12,-7/12-8);  \draw(2,-2/8)--(40/12,-5/12-8); \draw(4,-3/6)--(4,-3/6-8);
     
  \draw(56/12,-7/12)--(72/12,-9/12-8);\draw(72/12,-9/12)--(88/12,-11/12-8);\draw(88/12,-11/12)--(8,-3);

   \draw [fill] (8+0.8,-1+0.4) circle [radius=0.1];      \node[above  ] at (8+0.8,-1+0.4) {$u^{+}_9$};           \draw [fill] (8+0.8,-1+0.4-8) circle [radius=0.1];  \node[below ] at (8+0.8,-1+0.4-8) {$u^{-}_9$}; 

\draw [fill] (8+2.5*0.8,-1+2.5*0.4) circle [radius=0.1];      \node[above   ] at (8+2.5*0.8,-1+2.5*0.4) {$u^{+}_{10}$};           \draw [fill] (8+2.5*0.8,-1+2.5*0.4-8) circle [radius=0.1];  \node[below ] at (8+2.5*0.8,-1+2.5*0.4-8) {$u^{-}_{10}$}; 

\draw [fill] (8+4*0.8,-1+4*0.4) circle [radius=0.1];      \node[above ] at (8+4*0.8-0.2,-1+4*0.4) {$u^{+}_{11}$};           \draw [fill] (8+4*0.8,-1+4*0.4-8) circle [radius=0.1];  \node[below ] at (8+4*0.8+0.2,-1+4*0.4-8) {$u^{-}_{11}$}; 

\draw [fill] (8+5.1*0.8,-1+5.1*0.4) circle [radius=0.1];      \node[above ] at (8+5.1*0.8,-1+5.1*0.4) {$u^{+}_{12}$};           \draw [fill] (8+5.1*0.8,-1+5.1*0.4-8) circle [radius=0.1];  \node[below ] at (8+5.1*0.8,-1+5.1*0.4-8) {$u^{-}_{12}$};

    \draw [fill] (8+6.2*0.8,-1+6.2*0.4) circle [radius=0.1];      \node[above ] at (8+6.2*0.8,-1+6.2*0.4) {$u^{+}_{13}$};           \draw [fill] (8+6.2*0.8,-1+6.2*0.4-8) circle [radius=0.1];  \node[below ] at (8+6.2*0.8,-1+6.2*0.4-8) {$u^{-}_{13}$}; 
  
\draw [fill] (8+7.5*0.8,-1+7.5*0.4) circle [radius=0.1];      \node[above ] at (8+7.5*0.8,-1+7.5*0.4) {$u^{+}_{14}$};           \draw [fill] (8+7.5*0.8,-1+7.5*0.4-8) circle [radius=0.1];  \node[below ] at (8+7.5*0.8,-1+7.5*0.4-8) {$u^{-}_{14}$}; 

\draw [fill] (8+9*0.8,-1+9*0.4) circle [radius=0.1];      \node[above ] at (8+9*0.8,-1+9*0.4) {$u^{+}_{15}$};           \draw [fill] (8+9*0.8,-1+9*0.4-8) circle [radius=0.1];  \node[below ] at (8+9*0.8,-1+9*0.4-8) {$u^{-}_{15}$};

    \draw [fill] (8,-7) circle [radius=0.1];      \node[above left] at (8,-7) {$w_{4}$};       \draw [fill] (8+5.1*0.8,-1+5.1*0.4-4) circle [radius=0.1];  \node[above left ] at (8+5.1*0.8,-1+5.1*0.4-4) {$w_{5}$};      \draw [fill] (16,3-2) circle [radius=0.1];  \node[above right ] at (16,3-2) {$w_{6}$};

\draw(8,-7)--(8+0.8,-1+0.4-8); \draw(8+0.8,-1+0.4)--(8+0.8*2.5,-1+2.5*0.4-8); \draw(8+0.8*2.5,-1+2.5*0.4)--(8+4*0.8,-1+4*0.4-8);\draw(8+5.1*0.8,-1+5.1*0.4)--(8+5.1*0.8,-1+5.1*0.4-8);

\draw(8+4*0.8,-1+4*0.4)--(8+6.2*0.8,-1+6.2*0.4-8); \draw(8+6.2*0.8,-1+6.2*0.4)--(8+7.5*0.8,-1+7.5*0.4-8);       \draw (8+7.5*0.8,-1+7.5*0.4) --(8+9*0.8,-1+9*0.4-8);

\draw (8+9*0.8,-1+9*0.4) --(16,1);

 \draw [fill] (0,-2) circle [radius=0.1];          \node[above  left] at (0,-2) {$w_{10}$};    \draw [fill] (8,-5) circle [radius=0.1];\node[above  left] at (8,-5) {$w_9$};  
    \draw [fill] (0,-4) circle [radius=0.1];          \node[above  left] at (0,-4) {$w_{11}$};    \draw [fill] (16,-1) circle [radius=0.1];  \node[above right ] at (16,-1) {$w_{7}$}; 
     \draw [fill] (16,-3) circle [radius=0.1];  \node[above right ] at (16,-3) {$w_{8}$};
    \end{tikzpicture}  
  \caption{Fundamental domain for $(\Gamma_{s}^{(11)})_{\infty}$ in Heisenberg group}
\end{figure}

\section{To find the spinal spheres that cover the prism}

In this section, we will determine the generators of the sister of Euclidean Picard groups.
Our method is based on the special feature that the orbifold $\mathbf{H}^2_\mathbb{C}/\Gamma^{(d)}_s$ has only one cusp for $d=2,7,11$.  In the previous section, we found suitable generators of the stabliser and constructed a fundamental domain. We will show that adjoining  a reflection $I^{(d)}_0$ to $(\Gamma^{(d)}_s)_\infty$ gives the generators of the sister of the Euclidean Picard modular groups $\Gamma^{(d)}_s$. As in the proof of Theorem 3.5 of \cite{fp}, it is sufficient to show  that $\langle R^{(d)}_1,R^{(d)}_2,R^{(d)}_3,T^{(d)},I^{(d)}_0\rangle $ has only one cusp.  The key step is to  show the union of the boundaries of these isometric spheres in Heisenberg group contains the fundamental domains for the stabilizers. In order to achieve this goal,  the fundamental domains for the stabilizers  will  be decomposed into several pieces as polyhedra such that each polyhedron lies inside a spinal sphere.

Given the map
$$I^{(d)}_0=\left[\begin{array}{ccc}0&0&i/\sqrt{d}\\0&1&0\\i\sqrt{d}&0&0\end{array}\right].$$
The isometric sphere $\mathcal{B}_0^{(d)}$ of $I^{(d)}_0$ is a Cygan sphere centered at  $o=(0,0,0)$ with radius $\frac{2}{\sqrt{d}}$. Observe that  $I^{(d)}_0$
maps $\mathcal{B}^{(d)}_0$ to itself and swaps the inside and the outside of $\mathcal{B}_0^{(d)}$. The boundary of $\mathcal{B}_0^{(d)}$ is a spinal sphere denoted by 
\begin{equation}\label{eq:5-1}\mathcal{S}_0^{(d)}=\left\{(\zeta,t):\left||\zeta|^2+it\right|=2/\sqrt{d}\right\}.\end{equation}

Indeed we only need to consider the boundaries of isometric spheres in Heisenberg group because two isometric spheres have a
 non-empty interior intersection if and only if the boundaries have a non-empty interior intersection.

\subsection{The proof of Theorem 5.1}In this case, it is easy to see that the six vertices of the prism $\Sigma_2$  lie outside of $\mathcal{S}_0^{(2)}$
Therefore, we need to find more isometric spheres whose boundaries together with $\mathcal{S}_0^{(2)}$ contain the prism $\Sigma_2$.

We consider the map $$I^{(2)}_0R^{(2)}_3I^{(2)}_0=\left(\begin{array}{ccc}
-1 & 0& 0\\
-2 & -1& 0\\
2 & 2& -1
\end{array}\right).$$  $\mathcal{B}^{(2)}_1$ denotes its isometric sphere which is a Cygan sphere centered at the point $(1,0)$ with radius 1. 
The boundary of $\mathcal{B}^{(2)}_1$ is given by 

\begin{equation}\label{eq:5-2}\mathcal{S}_1^{(2)}=\left\{(z,t):\left||z-1|^2+it+2i\Im z \right|=1\right\}.\end{equation}
We need to consider $\mathcal{S}_0^{(2)}$ and several images of $\mathcal{S}_1^{(2)}$ under some suitable elements in $(\Gamma^{(2)}_s)_\infty$.  
In Heisenberg coordinates these are given by 

\begin{align*}
T^{(2)}(\mathcal{S}_1^{(2)})&=\left\{(z,t):\left||z-1|^2+it-i\sqrt{2}+2i\Im z \right|=1\right\},\\
{T^{(2)}}^{-1}(\mathcal{S}_1^{(2)})&=\left\{(z,t):\left||z-1|^2+it+i\sqrt{2}+2i\Im z \right|=1\right\},\\
{T^{(2)}}^{-1}R^{(2)}_1(\mathcal{S}_1^{(2)})&=\left\{(z,t):\left||z+1|^2+it+i\sqrt{2}-2i\Im z \right|=1\right\},\\
\end{align*}

We claim that the prism $\Sigma_2$ lies inside the union of $$\mathcal{S}_0^{(2)},\mathcal{S}_1^{(2)},T^{(2)}(\mathcal{S}_1^{(2)}),{T^{(2)}}^{-1}(\mathcal{S}_1^{(2)}),{T^{(2)}}^{-1}R^{(2)}_1(\mathcal{S}_1^{(2)});$$ see
Figure 5. for viewing these spinal spheres.

\begin{figure}
\begin{center}  
\includegraphics[height=3in]{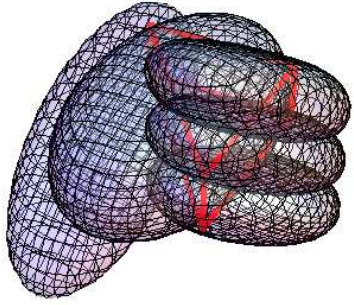}  
\caption{\small \sl The view of neighboring spinal spheres containing the fundamental domain for $(\Gamma_{s}^{(2)})_{\infty}$.}  
\end{center}  
\end{figure}

\begin{proposition}The prism $\mathbf{\Sigma}_2$ is contained in the union of the interiors of the spinal spheres  
$\mathcal{S}_0^{(2)},\mathcal{S}_1^{(2)},T^{(2)}(\mathcal{S}_1^{(2)}),{T^{(2)}}^{-1}(\mathcal{S}_1^{(2)}),{T^{(2)}}^{-1}R^{(2)}_1(\mathcal{S}_1^{(2)}).$
\end{proposition}

\begin{proof} 
It suffices to show that the prism $\Sigma_2$ can be decomposed into several pieces as polyhedra such that each polyhedron lies inside a spinal sphere which
is described in the proposition and the common face of two adjacent polyhedra lie in the intersection of the interior of two spinal spheres which contain these two polyhedra.

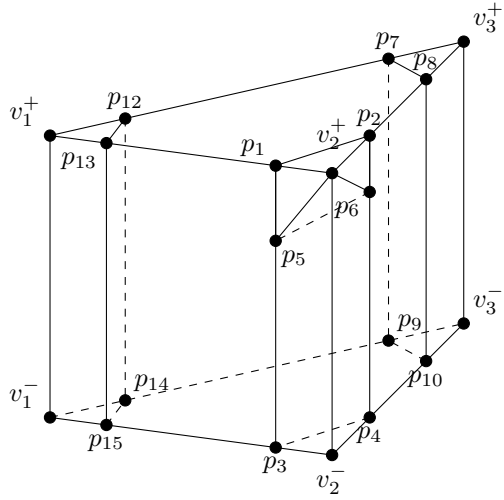
\begin{figure}
\begin{tikzpicture}[x=.5cm,y=.5cm]
\draw (0,0)--(11,2.5)--(7.5,-1)--cycle;
\draw (0,-7.5)--(7.5,-8.5)--(11,-5);
\draw[style=dashed] (0,-7.5) -- (11,-5);   
\draw (0,-7.5)--(0,0);  
\draw (11,-5)--(11,2.5);  
  \draw (7.5,-8.5)--(7.5,-1);    
\draw [fill] (0,0) circle [radius=0.15];
 \draw [fill] (0,-7.5) circle [radius=0.15];   
\draw [fill] (11,2.5) circle [radius=0.15];  \draw [fill] (11,-5) circle [radius=0.15]; 
 \draw [fill] (7.5,-1) circle [radius=0.15];   \draw [fill] (7.5,-8.5) circle [radius=0.15];

\draw [fill] (2,10/22) circle [radius=0.15]; \draw [fill] (2,10/22-7.5) circle [radius=0.15];  \draw[style=dashed] (2,10/22) -- (2,10/22-7.5);  
\draw [fill] (1.5,-3/15) circle [radius=0.15]; \draw [fill] (1.5,-3/15-7.5) circle [radius=0.15];  \draw (1.5,-3/15) -- (1.5,-3/15-7.5);  
  \draw[style=dashed] (1.5,-3/15-7.5) -- (2,10/22-7.5);   \draw (2,10/22) -- (1.5,-3/15);   
 \node[above ] at (2,10/22) {$p_{12}$};    \node[above right] at (2,10/22-7.5) {$p_{14}$};    
   \node[below left] at (1.5,-3/15) {$p_{13}$};  \node[below ] at (1.5,-3/15-7.5) {$p_{15}$};

\draw [fill] (9,45/22) circle [radius=0.15]; \draw [fill] (9,45/22-7.5) circle [radius=0.15];   \draw[style=dashed] (9,45/22) -- (9,45/22-7.5);    
\draw [fill] (10,1.5) circle [radius=0.15]; \draw [fill] (10,1.5-7.5) circle [radius=0.15];   \draw(10,1.5) -- (10,1.5-7.5);      
  \draw(10,1.5) -- (9,45/22);   \draw[style=dashed] (10,1.5-7.5) -- (9,45/22-7.5);   
  \node[above] at (9,45/22) {$p_{7}$};    \node[above right] at (9,45/22-7.5) {$p_{9}$};   
    \node[above ] at (10,1.5) {$p_{8}$};    \node[below] at (10,1.5-7.5) {$p_{10}$};    
  
\draw [fill] (6,-12/15) circle [radius=0.15]; \draw [fill] (6,-12/15-7.5) circle [radius=0.15];  \draw (6,-12/15) -- (6,-12/15-7.5);    
\draw [fill] (8.5,0) circle [radius=0.15]; \draw [fill] (8.5,0-7.5) circle [radius=0.15];  \draw (8.5,0) -- (8.5,0-7.5);    
    \draw(8.5,0) -- (6,-12/15);   \draw[style=dashed] (8.5,0-7.5) -- (6,-12/15-7.5);   
  \node[above left] at (6,-12/15) {$p_{1}$};     \node[below] at (6,-12/15-7.5) {$p_{3}$};  
     \node[above ] at (8.5,0) {$p_{2}$};   \node[below] at (8.5,0-7.5) {$p_{4}$};

 \draw [fill] (6,-12/15-2) circle [radius=0.15];   \draw (6,-12/15) -- (6,-12/15-2);  
  \draw [fill] (8.5,0-1.5) circle [radius=0.15];  \draw (8.5,0) -- (8.5,0-1.5);  
  \draw[style=dashed] (8.5,0-1.5) -- (6,-12/15-2);  
   \draw (7.5,-1) -- (6,-12/15-2);    \draw (7.5,-1) -- (8.5,0-1.5);  
  \node[below left] at (8.5,0-1.5) {$p_{6}$};  \node[below] at (6.5,-12/15-2) {$p_{5}$};     
   
 \node[above left] at (0,0) {$v^{+}_1$};   \node[above left] at (0,-7.5) {$v^{-}_1$};  
  \node[above ] at (7.5,-0.6) {$v^{+}_2$};   \node[below] at (7.5,-8.5) {$v^{-}_2$};   
 \node[above right] at (11,2.5) {$v^{+}_3$};     \node[above right] at (11,-5) {$v^{-}_3$};  
    
  \end{tikzpicture}  
  \caption{The decomposition of the fundamental domain for$(\Gamma_{s}^{(2)})_{\infty}$ into   several polyhedra. }
  
  \end{figure}

We need to add fifteen points on the faces of the prism $\Sigma_2$ in order to decompose the prism into five polyhedra, in Heisenberg coordinates. See Figure 6 for the decomposition of the prism. These are given by

$$
\begin{array}{llll}
p_1=(4/5-i2\sqrt{2}/5,\sqrt{2}/2),&p_2=(1-i3\sqrt{2}/10,\sqrt{2}/2),\\
p_3=(4/5-i2\sqrt{2}/5,-\sqrt{2}/2),&p_4=(1-i3\sqrt{2}/10,-\sqrt{2}/2),\\
p_5=(4/5-i2\sqrt{2}/5,2\sqrt{2}/5),&p_6=(1-i3\sqrt{2}/10,\sqrt{2}/5),\\
p_7=(4/5+i\sqrt{2}/2,\sqrt{2}/2),&p_8=(1+i3\sqrt{2}/10,-\sqrt{2}/2),\\
p_9=(4/5+i\sqrt{2}/2,-\sqrt{2}/2),&p_{10}=(1+i3\sqrt{2}/10,-\sqrt{2}/2),\\
p_{11}=(1+i\sqrt{2}/2,-2\sqrt{2}/5),&p_{12}=(-4/5+i\sqrt{2}/2,\sqrt{2}/2),\\
p_{13}=(-9/10+i9\sqrt{2}/20,\sqrt{2}/2),&p_{14}=(-4/5+i\sqrt{2}/2,-\sqrt{2}/2),\\
p_{15}=(-9/10+i9\sqrt{2}/20,-\sqrt{2}/2).
\end{array}$$

We verify the location of all these points as follows

$\bullet$\  the points $p_1, p_2$ are in the intersection of the interiors of $\mathcal{S}_0^{(2)}$ and $\mathcal{S}_1^{(2)}$;

$\bullet$\  the points $p_5, p_6$ are in the intersection of the interiors of $\mathcal{S}_0^{(2)},\mathcal{S}_1^{(2)}$ and ${T^{(2)}}^{-1}(\mathcal{S}_1^{(2)})$;

$\bullet$\  the points $p_3, p_4$ are in the intersection of the interiors of $\mathcal{S}_0^{(2)}$ and ${T^{(2)}}^{-1}(\mathcal{S}_1^{(2)})$;

$\bullet$\  the point $v_2^{+}$ is in the intersection of the interiors of $\mathcal{S}_1^{(2)}$ and ${T^{(2)}}^{-1}(\mathcal{S}_1^{(2)})$;

$\bullet$\  the point $v_2^{-}$ is in  the interior of ${T^{(2)}}^{-1}(\mathcal{S}_1^{(2)})$;

$\bullet$\  the point $v_3^{+}$ is in  the interior of $T^{(2)}(\mathcal{S}_1^{(2)})$;

$\bullet$\  the point $v_3^{-}$ is in the intersection of the interiors of $T^{(2)}(\mathcal{S}_1^{(2)})$ and $\mathcal{S}_1^{(2)}$;

$\bullet$\  the points $p_7,p_8$ are in the intersection of the interiors of $\mathcal{S}_0^{(2)}$ and $T^{(2)}(\mathcal{S}_1^{(2)})$;

$\bullet$\  the points $p_9,p_{10}$ are in the intersection of the interiors of $\mathcal{S}_0^{(2)}, T^{(2)}(\mathcal{S}_1^{(2)})$ and $\mathcal{S}_1^{(2)}$;

$\bullet$\  the point $p_{11}$ is in the intersection of the interiors of $\mathcal{S}_1^{(2)}$ and $T^{(2)}(\mathcal{S}_1^{(2)})$;

$\bullet$\  the points $v_1^{-},v_1^{+}$ are in  the interiors of ${T^{(2)}}^{-1}R^{(2)}_1(\mathcal{S}_1^{(2)})$;

$\bullet$\  the points $p_{12},p_{13},p_{14},p_{15}$ are in the intersection of the interiors of $\mathcal{S}_0^{(2)}$ and ${T^{(2)}}^{-1}R^{(2)}_1(\mathcal{S}_1^{(2)})$.

We describe these polyhedra as follows

$\bullet$\ the first polyhedra $\mathbb{P}_1^{(2)}$ with vertices $v_1^{+},p_{12}, p_{13}, v_1^{-},p_{14},p_{15}$;

$\bullet$\ the second polyhedra $\mathbb{P}_2^{(2)}$ with vertices $v_3^{+},p_{7}, p_{8}, v_3^{-},p_{9},p_{10}$;

$\bullet$\ the third polyhedra $\mathbb{P}_3^{(2)}$ with vertices $p_{12}, p_{13}, p_{2},p_{8},p_7,p_{14},p_{15},p_{3},p_{4},p_{10},p_9$;

$\bullet$\ the forth polyhedra $\mathbb{P}_4^{(2)}$ with vertices $v_2^{+},p_1,p_2,p_6,p_5$;

$\bullet$\ the fifth polyhedra $\mathbb{P}_5^{(2)}$ with vertices  $v_2^{+},p_5,p_6,v_2^{-},p_3,p_4$.

As the spinal sphere is convex,  we conclude that the polyhedra $\mathbb{P}_1^{(2)}$ is inside the spinal sphere ${T^{(2)}}^{-1}R^{(2)}_1(\mathcal{S}_1^{(2)})$;
the polyhedra $\mathbb{P}_2^{(2)}$ is inside the spinal sphere $T^{(2)}(\mathcal{S}_1^{(2)})$; the polyhedra $\mathbb{P}_3^{(2)}$ is inside the spinal sphere $\mathcal{S}_0^{(2)}$;
the polyhedra $\mathbb{P}_4^{(2)}$ is inside the spinal sphere$\mathcal{S}_1^{(2)}$; the polyhedra $\mathbb{P}_5^{(2)}$ is inside the spinal sphere ${T^{(2)}}^{-1}(\mathcal{S}_1^{(2)})$.
\end{proof}

\subsection{The proof of Theorem 1.2}

As in the case of $\Sigma_2$, the fundamental domain $\Sigma_7$ for the stabilizer  $(\Gamma_{s}^{(7)})_{\infty}$ also can not be contained inside $\mathcal{S}_0^{(7)}$ completely. Therefore we consider the maps

$$Q_1=I^{(7)}_0R^{(7)}_3I^{(7)}_0=\left(\begin{array}{ccc}
-1 & 0& 1\\
-7/2+\sqrt{7}i/2 & -1& 0\\
7 & 7/2+\sqrt{7}i/2& -1
\end{array}\right), $$ and $$ Q_2={R^{(7)}_2}^{-1}I^{(7)}_0R^{(7)}_2=\left(\begin{array}{ccc}
1/2-\sqrt{7}i/2& -1/2-\sqrt{7}i/2& 1+3i/\sqrt{7}\\
i\sqrt{7} & 1+i\sqrt{7}& -3/2-\sqrt{7}i/2\\
i\sqrt{7} & i\sqrt{7}& -1/2-\sqrt{7}i/2
\end{array}\right).$$

Consider the isometric spheres of $Q_1$ and $Q_2$, which are denoted  by $\mathcal{B}_1^{(7)}$ and $\mathcal{B}_2^{(7)}$, respectively. The center of 
$\mathcal{B}_1^{(7)}$ is $Q_1^{-1}(\infty)$, which is point with horospherical coordinate $(1/2,-1/2\sqrt{7},0)$ and the centre of  $\mathcal{B}_2^{(7)}$, is $Q_2^{-1}(\infty)$ which has horospherical coordinate
$(1,0,-1/\sqrt{7})$.  The  Cygan radius  of $\mathcal{B}_1^{(7)}$ and $\mathcal{B}_2^{(7)}$  are $\sqrt{2}/\sqrt{7}$ and  $\sqrt{2}/\sqrt[4]{7}$  respectively. The boundaries of these isometric spheres $\mathcal{B}_1^{(7)}$ and $\mathcal{B}_2^{(7)}$ are in Heisenberg coordinates given by 

\begin{align*}
\mathcal{S}_1^{(7)}&=\left\{(z,t):\left||z+i\omega_7/\sqrt{7}|^2+it+2i\Im (iz\overline{\omega_7}/\sqrt{7})\right|=2/7\right\},\\
\mathcal{S}_2^{(7)}&=\left\{(z,t):\left||z-1|^2+it+i/\sqrt{7}+2i\Im z \right|=2/\sqrt{7}\right\}
\end{align*}

In order to cover the prism $\Sigma_7$, we consider $\mathcal{S}_0^{(7)}$ and the images of $\mathcal{S}_1^{(7)} $  and  $\mathcal{S}_2^{(7)}$ under some elements in
$(\Gamma_{s}^{(7)})_{\infty}$. These spinal spheres are points with Heisenberg coordinates given by

\begin{align*}
T^{(7)}R^{(7)}_2(\mathcal{S}_1^{(7)})&=\left\{(z,t):\left||z-i\overline{\omega_7}/\sqrt{7}|^2+it-i2/\sqrt{7}+2i\Im (-iz\omega_7/\sqrt{7}) \right|=2/7\right\},\\
{T^{(7)}}^{2}(\mathcal{S}_2^{(7)})&=\left\{(z,t):\left||z-1|^2+it-i3/\sqrt{7}+2i\Im z \right|=2/\sqrt{7}\right\},\\
{T^{(2)}}^{3}(\mathcal{S}_2^{(7)})&=\left\{(z,t):\left||z-1|^2+it-i5/\sqrt{7}-2i\Im z \right|=2/\sqrt{7}\right\}.
\end{align*}

We claim that the prism $\Sigma_7$ lies inside the union of $$\mathcal{S}_0^{(7)},\mathcal{S}_1^{(7)},\mathcal{S}_2^{(7)}, T^{(7)}R^{(7)}_2(\mathcal{S}_1^{(7)}),{T^{(7)}}^{2}(\mathcal{S}_2^{(7)}),{T^{(7)}}^{3}(\mathcal{S}_2^{(7)}).$$ See
Figure 7 for viewing these spinal spheres.

\begin{figure}
\begin{center}  
\includegraphics[height=3in]{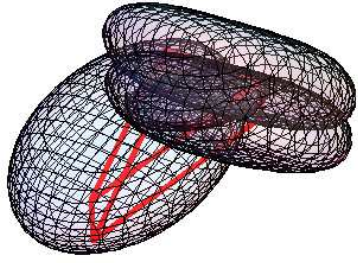}  
\caption{\small \sl The view of neighboring spinal spheres containing the fundamental domain for $(\Gamma_{s}^{(7)})_{\infty}$}  
\end{center}  
\end{figure}

\begin{proposition}The prism $\mathbf{\Sigma}_7$ is contained in the union of the interiors of the spinal spheres  
$\mathcal{S}_0^{(7)},\mathcal{S}_1^{(7)},\mathcal{S}_2^{(7)}, T^{(7)}R^{(7)}_2(\mathcal{S}_1^{(7)}),{T^{(7)}}^{2}(\mathcal{S}_2^{(7)}),{T^{(7)}}^{3}(\mathcal{S}_2^{(7)})$.
\end{proposition}

\begin{proof}

We need to add sixteen points on the faces of the prism $\mathbf{\Sigma}_7$ in order to decompose the prism into seven polyhedra, in Heisenberg coordinates, these are given by
$$
\begin{array}{llll}
p_1=(0.67,0.17\times\sqrt{17},\sqrt{7}/7),&p_2=(0.46,\sqrt{7}/4,\sqrt{7}/7),\\
p_3=(0.67,0.17\times\sqrt{17},-\sqrt{7}/7),&p_4=(0.46,\sqrt{7}/4,\sqrt{7}/7),\\
p_5=(3/4,\sqrt{7}/4,\sqrt{7}/8),&p_6=(0.46,\sqrt{7}/4,0),\\
p_7=(0.55,\sqrt{7}/4,\sqrt{7}/20),&p_8=(0.55,\sqrt{7}/4,\sqrt{7}/7),\\
p_9=(0.6,0.2\times\sqrt{7},\sqrt{7}/7),&p_{10}=(0.65,\sqrt{7}/4,7\sqrt{7}/80),\\
p_{11}=(0.6,0.2\times\sqrt{7},-\sqrt{7}/7).
\end{array}$$

We verify the location of all these points as follows

$\bullet$\  the points $p_1,p_3,p_4,p_5,p_6,p_9,p_{11},v_2^{-}$ are in the intersection of the interiors of $\mathcal{S}_0^{(7)}$ and ${T^{(7)}}^{2}(\mathcal{S}_2^{(7)})$;

$\bullet$\  the points $p_1,p_5,p_9,v_2^{+}$ are in  the interior of ${T^{(7)}}^{3}(\mathcal{S}_2^{(7)})$;

$\bullet$\ the point $p_9$ is in  the interior of $T^{(7)}R^{(7)}_2(\mathcal{S}_1^{(7)})$;

$\bullet$\  the point $p_2$ is in the intersection of the interiors of $\mathcal{S}_0^{(7)}$ and $T^{(7)}R^{(7)}_2(\mathcal{S}_1^{(7)})$;

$\bullet$\  the point $p_8$ is in the intersection of the interiors of ${T^{(7)}}^{3}(\mathcal{S}_2^{(7)})$ and $T^{(7)}R^{(7)}_2(\mathcal{S}_1^{(7)})$;

$\bullet$\  the point $p_{10}$ is in the intersection of the interiors of $T^{(7)}R^{(7)}_2(\mathcal{S}_1^{(7)})$, ${T^{(7)}}^{2}(\mathcal{S}_2^{(7)})$ and ${T^{(7)}}^{3}(\mathcal{S}_2^{(7)})$;

$\bullet$\  the point $p_7$ is in the intersection of the interiors of $T^{(7)}R^{(7)}_2(\mathcal{S}_1^{(7)})$, ${T^{(7)}}^{2}(\mathcal{S}_2^{(7)})$ and $\mathcal{S}_0^{(7)}$.

We describe these polyhedra as follows:

 $\bullet$\ the  polyhedra $\mathbb{P}_1^{(7)}$ with vertices  $v^+_1,p_1,p_9,p_2,v_3^{+}, v_1^{-},p_3,p_{11},p_4,v_3^{-}$;

$\bullet$\ the  polyhedra $\mathbb{P}_2^{(7)}$ with vertices  $p_1,p_5,p_9,p_3,v_2^{-},p_{11}$;

$\bullet$\ the  polyhedra $\mathbb{P}_3^{(7)}$ with vertices  $p_1,p_5,v_2^{+},p_{9}$;

$\bullet$\ the  polyhedra $\mathbb{P}_4^{(7)}$ with vertices  $p_5,p_9,p_6,v_2^{-},p_{11},p_4$;

$\bullet$\ the  polyhedra $\mathbb{P}_5^{(7)}$ with vertices  $p_9,p_5,,p_{10},p_8, v_2^{+}$;

$\bullet$\ the  polyhedra $\mathbb{P}_6^{(7)}$ with vertices  $p_9,p_5,p_{10},p_7, p_6$;

$\bullet$\ the  polyhedra $\mathbb{P}_7^{(7)}$ with vertices  $p_9,p_2,p_7, p_6$;

$\bullet$\ the  polyhedra $\mathbb{P}_8^{(7)}$ with vertices  $p_9,p_{10},p_7, p_2,p_8$.

By examining the location of the points and applying the spinal sphere is convex, we conclude that the polyhedra $\mathbb{P}_1^{(7)}$ is inside the spinal sphere
$\mathcal{S}_0^{(7)}$; the polyhedra $\mathbb{P}_2^{(7)}$ is inside the spinal sphere ${T^{(7)}}^{2}(\mathcal{S}_2^{(7)})$; the polyhedra $\mathbb{P}_3^{(7)}$ is inside the spinal sphere ${T^{(7)}}^{3}(\mathcal{S}_2^{(7)})$;
the polyhedra $\mathbb{P}_4^{(7)}$ is inside the spinal sphere ${T^{(7)}}^{2}(\mathcal{S}_2^{(7)})$; the polyhedra $\mathbb{P}_5^{(7)}$ is inside the spinal sphere ${T^{(7)}}^{3}(\mathcal{S}_2^{(7)})$;
the polyhedra $\mathbb{P}_6^{(7)}$ is inside the spinal sphere ${T^{(7)}}^{2}(\mathcal{S}_2^{(7)})$; the polyhedra $\mathbb{P}_7^{(7)}$ is inside the spinal sphere $\mathcal{S}_0^{(7)}$;
 the polyhedra $\mathbb{P}_8^{(7)}$ is inside the spinal sphere $T^{(7)}R^{(7)}_2(\mathcal{S}_1^{(7)})$.

\end{proof}

\begin{figure}
\begin{tikzpicture}[x=.5cm,y=.5cm]
\draw (0,0)--(15,3)--(8,-2)--cycle;
\draw (0,-10)--(8,-12)--(15,-7);
\draw[style=dashed] (0,-10) -- (15,-7);   
\draw (0,-10)--(0,0);  
\draw (8,-2)--(8,-12);  
 \draw (15,3)--(15,-7);    
\draw [fill] (0,0) circle [radius=0.15];
 \draw [fill] (0,-10) circle [radius=0.15];   
\draw [fill] (15,3) circle [radius=0.15];  \draw [fill] (15,-7) circle [radius=0.15]; 
 \draw [fill] (8,-2) circle [radius=0.15];   \draw [fill] (8,-12) circle [radius=0.15];
   
 \node[above left] at (0,0) {$v^{+}_1$};   \node[above left] at (0,-10) {$v^{-}_1$};  
  \node[above ] at (8.2,-1.9) {$v^{+}_2$};   \node[below] at (8,-12) {$v^{-}_2$};   
 \node[above right] at (15,3) {$v^{+}_3$};     \node[above right] at (15,-7) {$v^{-}_3$};

 \draw [fill] (4,-1) circle [radius=0.15];  \node[above left] at (4,-1) {$p_1$};\draw [fill] (4,-11) circle [radius=0.15]; \node[above left] at (4,-11) {$p_3$};
  \draw [fill] (11.5,0.5) circle [radius=0.15]; \node[above left] at (11.5,0.5) {$p_2$};\draw [fill] (11.5,-9.5) circle [radius=0.15];  \node[above left] at (11.5,-9.5) {$p_4$}; 
  
 \draw [fill] (8,-4.5) circle [radius=0.15];  \node[below right] at (8,-4.5) {$p_5$}; \draw [fill] (11.5,0.5-6) circle [radius=0.15];  \node[below right] at (11.5,0.5-6) {$p_6$}; 
 \draw [fill] (9.75,-0.75) circle [radius=0.15]; \node[above left] at (9.75,-0.75) {$p_8$}; 
 
  \draw [fill] (7.3,-0.6) circle [radius=0.15]; \node[above left] at (7.25,-0.3) {$p_{9}$}; 
   \draw [fill] (7.3,-0.6-10) circle [radius=0.15]; \node[above left] at (7.25,-0.3-10) {$p_{11}$}; 
      
  \draw [fill] (10.5,-0.75-2.7) circle [radius=0.15]; \node[above right] at (10.5,-0.75-2.7) {$p_7$};  
 
   \draw [fill] (9.25,-3.975) circle [radius=0.15]; \node[below] at (9.25,-3.975) {$p_{10}$};  
   
   \draw(4,-1)--(4,-11);
\draw(11.5,0.5)--(11.5,-9.5);        
 \draw(4,-1)--(8,-4.5);      
 
\draw(11.5,0.5)--(10.5,-0.75-2.7);\draw(11.5,0.5)--(11.5,0.5-6); 

\draw(11.5,-5.5)--(10.5,-0.75-2.7);   \draw(9.25,-3.975)--(10.5,-0.75-2.7);    

 \draw(9.25,-3.975)--(9.75,-0.75);   \draw(9.25,-3.975)--(8,-4.5); 
  \draw(8,-4.5)--(11.5,-5.5);

\draw[style=dashed](4,-1-10)--(7.3,-0.6-10);\draw[style=dashed](7.3,-0.6-10)--(11.5,0.5-10);\draw[style=dashed](7.3,-0.6-10)--(8,-2-10);
\draw(4,-1)--(7.3,-0.6);\draw(7.3,-0.6)--(11.5,0.5);\draw(7.3,-0.6)--(8,-2);\draw(7.3,-0.6)--(9.75,-0.75);
\draw[style=dashed](7.3,-0.6)--(7.3,-0.6-10);
\draw[style=dashed](7.3,-0.6)--(8,-4.5); 
 \draw[style=dashed](7.3,-0.6)--(9.25,-3.975);
\draw[style=dashed](7.3,-0.6)--(10.5,-0.75-2.7);   
\draw[style=dashed](7.3,-0.6)--(11.5,0.5-6);      
   \end{tikzpicture}  
  \caption{The decomposition of the fundamental domain for $(\Gamma_{s}^{(7)})_{\infty}$ in Heisenberg group}  
  \end{figure}
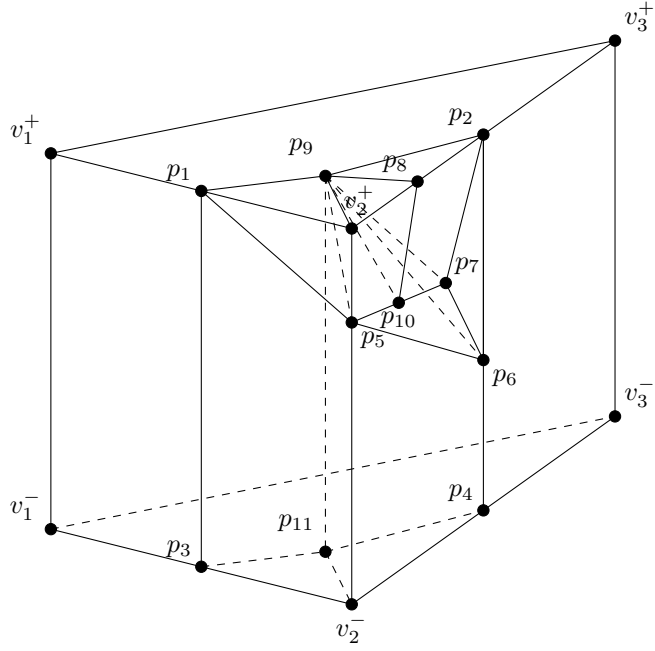

\subsection{The case $\mathcal{O}_{11}$}

In this case, it is also easy to know that the fundamental domain for the stabilizer $(\Gamma_{s}^{(11)})_{\infty}$ can not be inside
$\mathcal{S}_0^{(11)}$ completely. Moreover, the intersection of the intersection of the spinal spheres is more complicate than the above two cases.
So it is difficult to decompose $\Sigma_{11}$ into several polyhedral as before. We will apply the similar method in \cite{z2}. 
$\Sigma_{11}$ will be decomposed into several polyhedra cells. Observe that a union of intersection of several spinal spheres is a star-convex set if they
have a non-empty interior intersection. Thus we show that the collection of some spinal spheres can be separated into several parts such that
each part contains certain polyhedral cell. All of these polyhedral cells are defined by the cone-polygon as its boundary. The concept of cone-polygon was introduced in \cite{z2}.

First we consider the maps

$$A_1=(I^{(11)}_0R^{(11)}_2I^{(11)}_0)^{-1}\left(\begin{array}{ccc}
-1 & 0& 0\\
-i\sqrt{11}& -1& 0\\
11/2+i\sqrt{11}/2 & -\sqrt{11}i& -1
\end{array}\right)$$  and $$ A_2=({R^{(11)}_2}I^{(11)}_0)^2=\left(\begin{array}{ccc}
-7/2+3\sqrt{11}i/2 & -3/2-\sqrt{11}i/2& 1/2-i/2\sqrt{11}\\
11/2-3\sqrt{11}i/2 & 1+\sqrt{11}i& -1\\
11/2-\sqrt{11}i/2 & \sqrt{11}i& -1
\end{array}\right).$$

The boundaries of the isometric spheres of $A_1, A_2$ are in Heisenberg coordinates given by 

\begin{align*}
\mathcal{S}_1^{(11)}&=\left\{(z,t):\left||z+\overline{\omega_{11}}/3|^2+it+i/3\sqrt{11}+2i\Im (-z\omega_{11}/3)\right|=\frac{2}{\sqrt{33}}\right\},\\
\mathcal{S}_2^{(11)}&=\left\{(z,t):\left||z+\omega_{11}/3|^2+it-i/3\sqrt{11}+2i\Im (-z\overline{\omega_{11}}/3)\right|=\frac{2}{\sqrt{33}}\right\}.
\end{align*}

In order to determine a union of the spinal spheres which covers the prism $\Sigma_{11}$,  we need to consider $\mathcal{S}_0^{(11)}$ and the images of $\mathcal{S}_1^{(11)} $  and  $\mathcal{S}_2^{(11)}$ under some elements in
$(\Gamma_{s}^{(11)})_{\infty}$. We will show that the prism $\Sigma_{11}$ lies inside the union of ten spinal spheres. See
Figure 9 for viewing these spinal spheres.

\begin{proposition}The prism  $\Sigma_{11}$ is contained in the union of the interiors of the spinal spheres  
\begin{align*}
&\mathcal{S}_0^{(11)},\mathcal{S}_1^{(11)},T^{(11)}(\mathcal{S}_1^{(11)}),R^{(11)}_1T^{(11)}(\mathcal{S}_1^{(11)}), {T^{(11)}}^{3}R^{(11)}_2R^{(11)}_1(\mathcal{S}_1^{(11)})  \\
&R^{(11)}_1(\mathcal{S}_1^{(11)}),{T^{(11)}}^{2}R^{(11)}_2R^{(11)}_1(\mathcal{S}_1^{(11)}), T^{(11)}R^{(11)}_2R^{(11)}_1(\mathcal{S}_1^{(11)}),\\
&(R^{(11)}_1R^{(11)}_3)^{-1}(\mathcal{S}_2^{(11)}),T^{(11)}(R^{(11)}_1R^{(11)}_3)^{-1}(\mathcal{S}_2^{(11)}).
\end{align*}
\end{proposition}

\begin{proof}

We will use another fundamental domain of the stabilizer.  That is, a prism has vertices
$v_1^{\pm1}=(1/4-i\sqrt{11}/4,\pm \sqrt{11}/11), v_2^{\pm1}=(3/4+i\sqrt{11}/4,\pm \sqrt{11}/11), v_3^{\pm1}=(-1/4+i\sqrt{11}/4,\pm \sqrt{11}/11)$.
This   can be achieved by using   cut-and-paste technique. 
This makes determining the spinal spheres relatively easy for the us.

We need to add sixteen points on the faces of the prism $\Sigma_{11}$ in order to decompose the prism into nine polyhedra cells, in Heisenberg coordinates, these are given by
$$
\begin{array}{llll}
p_1=(0.2,0.2\times\sqrt{11},0.33166),&p_2=(0.2,0.2\times\sqrt{11},-0.27136),\\
p_3=(0.31,-0.630159,0.3196),&p_4=(0.31,-0.630159,0.3196,-0.283421),\\
p_5=(0.2,-0.663325,0.180907),&p_6=(1/4,-\sqrt{11}/4,2\sqrt{11}/33),\\
p_7=(0.31,-0.630159,0.0783929),&p_8=(-0.14,0.464327,0.53669),\\
p_9=(0.1,0.549808,0.512165),&p_{10}=(0.41,0.469386,0.454464),\\
p_{11}=(0.59,0.298496,0.404025),&p_{12}=(-0.23,0.729657,-0.0180907)\\
p_{13}=(0.1,0.759809,-0.0622211),&p_{14}=(0.62,0.397995,-0.01809)\\
p_{15}=(-0.1824,0.604952,0.242656),&p_{16}=(-1/4,\sqrt{11}/4,78\sqrt{11}/1100)\\
p_{17}=(0.6716,0.569133,0.428629),&p_{18}=(0.75,0.829156,0.063023)\\
p_{19}=(0.4676,0.414701,0.438324),&p_{20}=(0.578,0.829156,0.478197)\\
p_{21}=(0.35,\sqrt{11}/4,0.512569),&p_{22}=(0.08,\sqrt{11}/4,0.553273)\\
p_{23}=(0.5825,0.829156,-0.125504),&p_{24}=(0.3346,0.829156,-0.0881318)\\
p_{25}=(0.08,\sqrt{11}/4,-0.0497494),&p_{26}=(0.3,0.709156,-0.0992793)\\
p_{27}=(0.378801,0.656463,0.484678),&p_{28}=(0.2519,0.617501,0.190955)\\
p_{29}=(0.6179,0.507112,-0.148374),&p_{30}=(0.6733,0.574771,-0.173882)\\
p_{31}=(0.5264,0.463121,-0.166961),&p_{32}=(0.334,0.596992,-0.1197)\\
p_{33}=(0.4412,0.829156,0.49882),&p_{34}=(0.18184,0.829156,-0.0651023)\\
p_{33}=(0.15,0.747145,-0.0714856).
\end{array}$$

We focus on describing other polyhedral cells in the decomposition of the prism $\Sigma_{11}$. Let $\mathcal{U}_1$ denote the union of $\mathcal{S}_0^{(11)},{T^{(11)}}^{3}R^{(11)}_2R^{(11)}_1(\mathcal{S}_1^{(11)}),{T^{(11)}}^{2}R^{(11)}_2R^{(11)}_1(\mathcal{S}_1^{(11)})$. We verify that $p_{19}$ is in the intersection of the interiors of these three spinal spheres, which implies that $\mathcal{U}_1$ is a star-convex set about $p_{19}$. Analogously, we know $\mathcal{U}_2$, denoted by the union of $\mathcal{S}_1^{(11)},(R^{(11)}_1R^{(11)}_3)^{-1}(\mathcal{S}_2^{(11)}),T^{(11)}(R^{(11)}_1R^{(11)}_3)^{-1}(\mathcal{S}_2^{(11)})$, is a star-convex set about the point $(0.3719,0.7475,0.4459)$.

\begin{figure}
\begin{center}  
\includegraphics[height=3in]{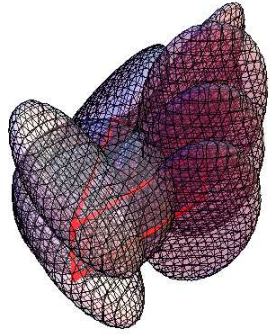}  
\caption{\small \sl The view of neighboring spinal spheres containing the fundamental domain for $(\Gamma_{s}^{(11)})_{\infty}$}  
\end{center}  
\end{figure}

We describe the nine polyhedra cells as follows:

 $\bullet$\ $\mathbb{P}_1^{(11)}$ with vertices  $v^+_1,p_1,p_3,p_5,p_6,p_7$;

$\bullet$\  $\mathbb{P}_2^{(11)}$ with vertices  $v^-_1,p_2,p_4,p_5,p_6,p_7$;

$\bullet$\  $\mathbb{P}_3^{(11)}$ with vertices  $p_1,p_3,p_{11},p_{19},p_{10},p_{9},p_{8},p_2,p_4,p_{14},p_{31},p_{32},p_{13},p_{12}$;

$\bullet$\  $\mathbb{P}_4^{(11)}$ with vertices  $v^+_3,p_8,p_9,p_{15},p_{16},p_{22}$;

$\bullet$\  $\mathbb{P}_5^{(11)}$ with vertices  $p_{10},p_{27},p_{21},p_{25},p_{13},p_{28},p_{9},p_{22},p_{12},v^-_{3},p_{15},p_{16}$;

$\bullet$\  $\mathbb{P}_6^{(11)}$ with vertices  $p_{14},p_{29},p_{31},p_{30}$;

$\bullet$\  $\mathbb{P}_7^{(11)}$ with vertices  $p_{11},p_{17},p_{19},p_{31},p_{30},p_{32},v^-_{2},p_{23},p_{18}$;

$\bullet$\  $\mathbb{P}_8^{(11)}$ with vertices  $p_{19},p_{17},v^+_2,p_{20},p_{27},p_{10},p_{28},p_{13},p_{26},p_{24},p_{23},p_{13},p_{32}$;

$\bullet$\  $\mathbb{P}_9^{(11)}$ with vertices  $p_{20},p_{33},p_{21},p_{27},p_{28},p_{13},p_{26},p_{24},p_{34},p_{25}$;

By examining the location of the points, we conclude that  $\mathbb{P}_1^{(11)}$ is inside the spinal sphere
$R^{(11)}_1T^{(11)}(\mathcal{S}_1^{(11)})$;  $\mathbb{P}_2^{(11)}$ is inside the spinal sphere
$R^{(11)}_1(\mathcal{S}_1^{(11)})$;  $\mathbb{P}_3^{(11)}$ is inside the spinal sphere
$\mathcal{S}_0^{(11)}$; $\mathbb{P}_4^{(11)}$ is inside the spinal sphere
$T^{(11)}(\mathcal{S}_1^{(11)})$; $\mathbb{P}_5^{(11)}$ is inside the spinal sphere
$\mathcal{S}_1^{(11)}$;  $\mathbb{P}_6^{(11)}$ is inside the spinal sphere
$T^{(11)}R^{(11)}_2R^{(11)}_1(\mathcal{S}_1^{(11)})$;  $\mathbb{P}_7^{(11)}$ is inside the spinal sphere
${T^{(11)}}^{2}R^{(11)}_2R^{(11)}_1(\mathcal{S}_1^{(11)})$; $\mathbb{P}_8^{(11)}$ is inside the star-convex set
$\mathcal{U}_1$;  $\mathbb{P}_9^{(11)}$ is inside the star-convex set
$\mathcal{U}_2$.

\end{proof}

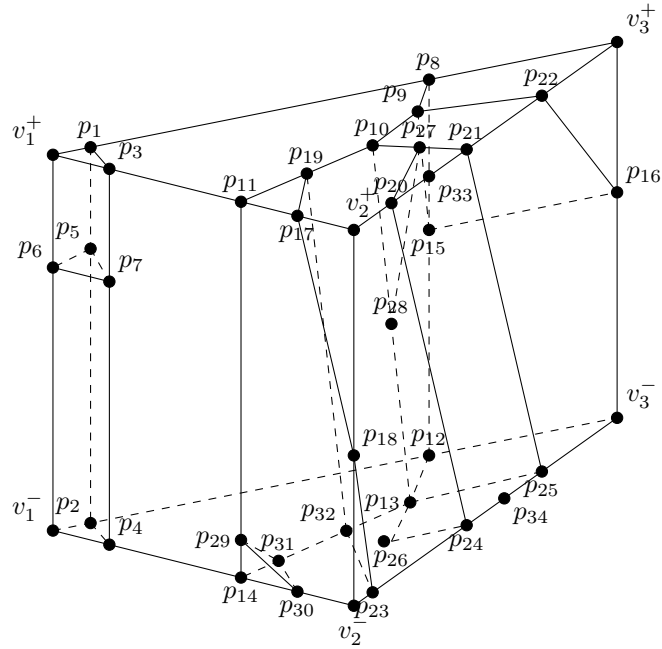
\begin{figure}
\begin{tikzpicture}[x=.5cm,y=.5cm]
\draw (0,0)--(15,3)--(8,-2)--cycle;
\draw (0,-10)--(8,-12)--(15,-7);
\draw[style=dashed] (0,-10) -- (15,-7);   
\draw (0,-10)--(0,0);  
\draw (8,-2)--(8,-12);  
 \draw (15,3)--(15,-7);    
\draw [fill] (0,0) circle [radius=0.15];
 \draw [fill] (0,-10) circle [radius=0.15];   
\draw [fill] (15,3) circle [radius=0.15];  \draw [fill] (15,-7) circle [radius=0.15]; 
 \draw [fill] (8,-2) circle [radius=0.15];   \draw [fill] (8,-12) circle [radius=0.15];
   
 \node[above left] at (0,0) {$v^{+}_1$};   \node[above left] at (0,-10) {$v^{-}_1$};  
  \node[above ] at (8.2,-1.9) {$v^{+}_2$};   \node[below] at (8,-12) {$v^{-}_2$};   
 \node[above right] at (15,3) {$v^{+}_3$};     \node[above right] at (15,-7) {$v^{-}_3$};

 \draw [fill] (0,-3) circle [radius=0.15];   \draw [fill] (1,1*3/15) circle [radius=0.15];     \draw [fill] (3/2,-3/8) circle [radius=0.15];   \draw [fill] (1,1*3/15-10) circle [radius=0.15];     \draw [fill] (3/2,-3/8-10) circle [radius=0.15]; 
 \draw [fill] (1,1*3/15-2.7) circle [radius=0.15];     \draw [fill] (3/2,-3/8-3) circle [radius=0.15];    
  \node[above ] at (1,3/15) {$p_1$};   \node[above left] at (1,3/15-10) {$p_2$};     \node[above right] at (3/2,-3/8) {$p_3$}; \node[above right ] at (3/2,-3/8-10) {$p_4$};   
  \node[above left] at (1,3/15-2.7) {$p_5$}; \node[above left] at (0,-3) {$p_6$};   \node[above right] at (3/2,-3/8-3) {$p_7$};         
 \draw[style=dashed] (1,3/15) -- (1,3/15-10);    \draw[style=dashed] (1,3/15-10) -- (3/2,-3/8-10);   
  \draw(3/2,-3/8)--(3/2,-3/8-10);\draw(1,3/15)--(3/2,-3/8);
  \draw[style=dashed] (1,3/15-2.7) -- (3/2,-3/8-3);    \draw[style=dashed] (1,3/15-2.7) -- (0,-3);     \draw(0,-3) -- (3/2,-3/8-3);

 \draw [fill] (10,10*3/15) circle [radius=0.15];    \draw [fill] (5,-5/4) circle [radius=0.15];      \draw [fill] (5+3.5,-5/4+1.5) circle [radius=0.15];    \draw [fill] (5+4.7,-5/4+2.4) circle [radius=0.15];    
 \draw(5,-5/4)--(5+3.5,-5/4+1.5); \draw(5+4.7,-5/4+2.4)--(5+3.5,-5/4+1.5);\draw(5+4.7,-5/4+2.4)--(10,10*3/15);
  \node[above ] at (10,10*3/15) {$p_8$};  \node[above left] at (5+4.7,-5/4+2.4) {$p_9$};  \node[above ] at (5+3.5,-5/4+1.5) {$p_{10}$};   \node[above ] at (5,-5/4) {$p_{11}$};             
        
 \draw [fill] (10,10*3/15-10) circle [radius=0.15];    \draw [fill] (5,-5/4-10) circle [radius=0.15];     \draw [fill] (5+4.5,-5/4+2-10) circle [radius=0.15];      
    \draw[style=dashed] (10,10*3/15-10)-- (5+4.5,-5/4+2-10);    \draw[style=dashed] (5+4.5,-5/4+2-10)--(5,-5/4-10);  
 \draw(5,-5/4) --(5,-5/4-10) ;  \draw[style=dashed] (10,10*3/15-10)-- (10,10*3/15);     \draw[style=dashed]  (5+4.5,-5/4+2-10)--(5+3.5,-5/4+1.5);     
      \node[above ] at (10,10*3/15-10) {$p_{12}$};       \node[left ] at (5+4.5,-5/4+2-10) {$p_{13}$};    \node[below ] at (5,-5/4-10) {$p_{14}$};

 \draw [fill] (10,10*3/15-4) circle [radius=0.15];    \node[below ] at (10,10*3/15-4) {$p_{15}$};   \draw [fill] (15,3-4) circle [radius=0.15];    \node[above right ] at (15,3-4) {$p_{16}$};    
   \draw [fill] (6.5,-1.625) circle [radius=0.15];    \node[below] at (6.5,-1.625){$p_{17}$};    \draw [fill] (8,-2-6) circle [radius=0.15];    \node[above right ] at (8,-2-6) {$p_{18}$};       
       \draw [fill] (6.75,-0.5) circle [radius=0.15];    \node[above  ] at (6.75,-0.5) {$p_{19}$};    
       
   \draw [fill] (9,5*1/7-2) circle [radius=0.15];   \node[above  ] at (9,5/7-2) {$p_{20}$};   \draw [fill] (11,5*3/7-2) circle [radius=0.15];   \node[above  ] at (11,5*3/7-2) {$p_{21}$};  
     \draw [fill] (13,5*5/7-2) circle [radius=0.15];  \node[above  ] at (13,5*5/7-2) {$p_{22}$};             
        \draw [fill] (8.5,5*0.5/7-2-10) circle [radius=0.15];  \node[below  ] at (8.5,5*0.5/7-2-10) {$p_{23}$}; 
           \draw [fill] (11,5*3/7-2-10) circle [radius=0.15];  \node[below  ] at (11,5*3/7-2-10) {$p_{24}$}; 
           \draw [fill] (13,5*5/7-2-10) circle [radius=0.15];      \node[below ] at (13,5*5/7-2-10) {$p_{25}$};            
             \draw [fill] (9-0.2,5*1/7-2-9) circle [radius=0.15];  \node[below] at (9,5/7-2-9) {$p_{26}$};  
      \draw [fill] (39/4,11/56) circle [radius=0.15];  \node[above  ] at (39/4,11/56) {$p_{27}$};  
       \draw [fill] (9,-9/2) circle [radius=0.15];  \node[above  ] at (9,-9/2) {$p_{28}$};      \draw [fill]  (5,-5/4-9)  circle [radius=0.15];   \node[left ] at (5,-5/4-9) {$p_{29}$};  
        \draw [fill]  (6.5,-93/8)  circle [radius=0.15];   \node[below ] at (6.5,-93/8)  {$p_{30}$};       
       
       \draw [fill]  (6,4/9-45/4)  circle [radius=0.15];   \node[above ] at (6,4/9-45/4)   {$p_{31}$};       
          \draw [fill]  (7.8,56/45-45/4)  circle [radius=0.15];   \node[above left ] at (7.8,56/45-45/4)    {$p_{32}$};    
          
     \draw [fill]  (10,10/7-2)  circle [radius=0.15];    \node[below right ] at (10,10/7-2)    {$p_{33}$};      
        \draw [fill]  (12,5*4/7-2-10)  circle [radius=0.15];    \node[below right ] at (12,5*4/7-2-10)   {$p_{34}$};

 \draw[style=dashed] (5,-5/4-9) --(6,4/9-45/4);   \draw(5,-5/4-9) --(6.5,-93/8) ;  \draw[style=dashed] (6.5,-93/8) --(6,4/9-45/4);                
    \draw(6.75,-0.5)--(6.5,-1.625);     \draw(8,-2-6)--(6.5,-1.625);    \draw(8,-2-6)--(8.5,5*0.5/7-2-10);    
    \draw[style=dashed] (7.8,56/45-45/4) --(6.75,-0.5);   \draw[style=dashed] (7.8,56/45-45/4)  --(8.5,5*0.5/7-2-10);  
   \draw (5+3.5,-5/4+1.5)  --(11,5*3/7-2);  
   \draw[style=dashed] (13,5*5/7-2-10)  --(5+4.5,-5/4+2-10);  
  \draw(11,5*3/7-2) --(13,5*5/7-2-10);   
     \draw(39/4,11/56)  --(9,5/7-2) ;   \draw(11,5*3/7-2-10) --(9,5/7-2) ;  \draw[style=dashed] (39/4,11/56)  -- (9,-9/2);     
 \draw[style=dashed] (9,5/7-2-9) --(5+4.5,-5/4+2-10);   \draw[style=dashed] (9,5/7-2-9) --(11,5*3/7-2-10);   
  
  \draw[style=dashed] (10,10*3/15-4) --(15,3-4); \draw[style=dashed] (10,10*3/15-4) --(5+4.7,-5/4+2.4) ; \draw (13,5*5/7-2)-- (15,3-4);   \draw (13,5*5/7-2)-- (5+4.7,-5/4+2.4);               
                                
   \end{tikzpicture}  
  \caption{The decomposition of the fundamental domain for $(\Gamma_{s}^{(11)})_{\infty}$ in Heisenberg group}
  
  \end{figure}

\section*{Acknowledgements}

The author would like to thank Professor
E. Falbel for   useful  discussions  during his visit to 
Pierre and Marie Curie University. I also thank Jieyan Wang for  discussion on this paper. 
This work was supported by  NSF(No.11201134 and No. 11071059).

\end{document}